\documentclass{article}

\usepackage{styles/VABcommands}
\usepackage{styles/VABenvironments}
\usepackage[numbers,sort&compress]{natbib}

\usepackage{thmtools}
\usepackage{thm-restate}
\usepackage{enumitem}
\usepackage{cancel}
\usepackage{tikz}
\usetikzlibrary{decorations.pathreplacing}

\newcommand\blfootnote[1]{%
  \begingroup
  \renewcommand\thefootnote{}\footnote{#1}%
  \addtocounter{footnote}{-1}%
  \endgroup
}

\setenumerate[1]{label={\arabic*)}} 

\title{Littlewood--Paley--Rubio de Francia inequality \protect\\ for multi-parameter Vilenkin systems}
\author{
Viacheslav Borovitskiy
\blfootnote{
\indent
The author's present affiliation is ETH Zürich, Zürich, Switzerland.}
\\[0.1cm]
\small{St. Petersburg Department of Steklov Mathematical Institute} \\* \small{of Russian Academy of Sciences (PDMI RAS)}
}
\date{\vspace{-5ex}}

\renewtheorem{definition}{Definition}

\begin{document}

\maketitle

\begin{abstract}
    A version of Littlewood--Paley--Rubio de Francia inequality for bounded multi-parameter Vilenkin systems is proved: for any family of disjoint sets $I_k = I_k^1 \x \ldots \x I_k^D \subseteq {\Z_+^D}$ such that $I_k^d$ are intervals in $\Z_+$ and a family of functions $f_k$ with Vilenkin--Fourier spectrum inside $I_k$ the following holds:
    $$
    \norm[2]{{\sum}_k f_k}_{L^p} \leq C
    \norm[2]{\del[2]{{\sum}_k \abs{f_k}^2}^{1/2}}_{L^p}
    , \qquad 1 < p \leq 2,
    $$
    where $C$ does not depend on the choice of rectangles $\cbr{I_k}$ or functions~$\cbr{f_k}$.
    This result belongs to a line of studying of (multi-parameter) generalizations of Rubio de Francia inequality to locally compact abelian groups.
    The arguments are mainly based on the atomic theory of multi-parameter martingale Hardy spaces and, as a byproduct, yield an easy-to-use multi-parameter version of Gundy's theorem on the boundedness of operators taking martingales to measurable functions.
    Additionally, some extensions and corollaries of the main result are obtained, including a weaker version of the inequality for exponents $0 < p \leq 1$ and an example of a one-parameter inequality for an exotic notion of the interval.
\end{abstract}

\section{Introduction}

\begin{samepage}
\paragraph{Motivation} In 1985 Rubio de Francia \cite{rubiodefrancia1985} proved that for an arbitrary family of intervals $\cbr{I_k}_{k \in \N}, I_k = [a_k, b_k) \subseteq \R$ with no intersections the following holds:
\[ \label{eqn:rdf_classic}
\norm[2]{\del[2]{{\sum}_{k = 1}^\infty \abs{M_{I_k} f}^2}^{1/2}}_{L^p(\R)}
\lesssim
\norm[2]{f}_{L^p(\R)}
, \qquad f \in L^p(\R), ~~p \geq 2,
\]
where each $M_{I_k}$ is the Fourier multiplier with symbol $\1_{I_k}$ and the sign $\lesssim$ means inequality with an implicit multiplicative constant.
In the modern literature the relation~\eqref{eqn:rdf_classic} is usually called \emph{Littlewood--Paley--Rubio de Francia inequality} or simply \emph{Rubio de Francia inequality}.
\end{samepage}

Rubio de Francia inequality relates the norm of a function with the norm of its ``pieces'' with restricted frequencies.
There are many results of this kind, the very basic one follows directly from the Plancherel theorem:
\[ \label{eqn:plancherel_cor}
\norm[1]{\del[1]{{\sum}_{k = 1}^\infty \abs{M_{I_k} f}^2}^{1/2}}_{L^2(\R)}
=
\norm[1]{f}_{L^2(\R)}
, \qquad f \in L^2(\R),
\]
where $I_k$, with $\cup I_k = \R$, may be arbitrary measurable sets that do not have pairwise intersections.
Since Equation~\eqref{eqn:plancherel_cor} does not hold for $L^p(\R)$ for $p \not= 2$, finer statements are needed to characterize the relationships between the sides of Equation~\eqref{eqn:plancherel_cor} when $L^2(\R)$ is substituted for $L^p(\R)$.
Rubio de Francia inequality is one of such results, and it is similar to and, in a certain sense, more general than yet another result of this kind, the famous Littlewood--Paley inequality.

As is often the case, there is a version of the inequality~\eqref{eqn:rdf_classic} for the circle $\T$ instead of the line $\R$.
To obtain this version we need to substitute $L^p(\R)$ for $L^p(\T)$, take $I_k = [a_k, b_k) = \Set{x \in \Z}{ a_k \leq x < b_k}$ to be arbitrary intervals in $\Z$ with no pairwise intersections and define $M_I f$ through the Fourier series instead of the Fourier transform: $M_{I_k} f = \sum_{n \in I_k} \innerprod{f}{\phi_n} \phi_n$ with $\phi_n(x)= e^{2 \pi i n x}$.

While transferring the inequality~\eqref{eqn:rdf_classic} to the circle is trivial (same proof works without a change), transferring it to other settings may be a difficult problem.
This is the case already for the multi-parameter situation of $\R^D$ and $\T^D$ instead of $\R$ and $\T$ with arbitrary rectangles in place of the intervals.

All the mentioned settings, as well as many others, are instances of an abstract notion of Rubio de Francia inequality defined for arbitrary locally compact abelian groups and partial orderings on their Pontryagin duals~(that induce the notion of an interval or rectangle) detailed in Section~\ref{sec:extensions}.
A natural question to ask is: for which groups with partially ordered duals does the inequality~hold?
This question seems difficult to resolve outright, at least before considering principal examples.
Motivated by this, we prove in this work the analog of the inequality~\eqref{eqn:rdf_classic} for Vilenkin groups giving rise to bounded Vilenkin systems and a natural notion of rectangles on the corresponding dual groups.

In the course of the proof, we verify some auxiliary results, most notably a multi-parameter version of Gundy's theorem and a somewhat sharpened version of it (in the spirit of~\cite{watari1958}).
Beyond proving the main result, we describe a technique which allows one to get similar inequalities for non-standard notions of the rectangle, prove a relevant no-go result and briefly touch the case of $p \leq 1$.

\paragraph{The main result} We prove the version of Rubio de Francia inequality for bounded multi-parameter Vilenkin systems.
First, we restrict our attention to the following formulation of the inequality~\eqref{eqn:rdf_classic} that is equivalent to it by duality:
\[ \label{eqn:rdf_classic_dual}
\norm[2]{{\sum}_{k=1}^\infty f_k}_{L^p(\R)} \lesssim \norm[2]{\del[2]{{\sum}_{k = 1}^\infty \abs{f_k}^2}^{1/2}}_{L^p(\R)}, \quad M_{I_k} f_k = f_k, ~~1 < p \leq 2.
\]
Starting from this formulation, we substitute the real line $\R$ for the $[0,1)^D$, assume that $I_k = I_k^1 \x \dots \x I_k^D$ are products of $D$ arbitrary intervals $I_k^d \subseteq \Z_+$ such that \mbox{$I_k \cap I_l = \emptyset$ for $k \not=l$}, and define operators $M_I$ for $I \subseteq \Z_+^D$ by expression~$M_I f = \sum_{n \in I} \innerprod{f}{\phi_n} \phi_n$, where $\phi_n$ are tensor products of Vilenkin functions.
That is we consider the analog of Rubio de Francia inequality corresponding to the multi-parameter Vilenkin--Fourier series instead of the Fourier transform.
We prove that this inequality does indeed hold for arbitrary bounded Vilenkin systems.
It is formalized in the following statement.

\begin{restatable}{theorem}{vilenkinrdftwodim}
\label{th:main_theorem}
Let $I_k = I_k^1 \x \ldots \x I_k^D \subseteq \Z_+^D$ be a family of disjoint sets such that~$I_k^d$, $d \in \cbr{1, 2, \ldots, D}$, are intervals in $\Z_+$.
Let $f_k: [0,1)^D \to \R$ be some functions with Vilenkin--Fourier spectrum in $I_k$, more precisely
\[
f_k(x) = {\sum}_{(n_1, \ldots, n_D) \in I_k} \innerprod{f_k}{v_{n_1} \otimes \ldots \otimes v_{n_D}} v_{n_1}(x_1) \cdot \ldots \cdot v_{n_D}(x_D),
\]
where $v_{n_d}$ are Vilenkin functions corresponding to a bounded Vilenkin system, possibly different for different values of $d$.
Then
\[
\norm[1]{{\sum}_k f_k}_{L^p} \leq C \norm[1]{\cbr{f_k}}_{L^p(l^2)}, \qquad 1 < p \leq 2,
\]
where $C$ does not depend on the choice of rectangles $\cbr{I_k}$ or functions $\cbr{f_k}$.
\end{restatable}

\paragraph{Prior work} Our result stands in line with a number of extensions and generalizations of the original inequality of Rubio de Francia.

Most of the extensions are dedicated to the trigonometric setting with $M_I$ defined through the usual Fourier transform or Fourier series.
For instance, inequality~\eqref{eqn:rdf_classic_dual} was proved first for $p = 1$ by Bourgain~\cite{bourgain1985} and then for arbitrary indices $p \leq 1$ by Kislyakov and Parilov~\cite{kislyakov2005}.
The multi-parameter version of inequalities \eqref{eqn:rdf_classic} and \eqref{eqn:rdf_classic_dual} was proved by Journe~\cite{journe1985} and later simplified by Soria~\cite{soria1987}, while Osipov proved the multi-parameter version of the inequality~\eqref{eqn:rdf_classic_dual} for arbitrary indices $p \leq 1$~\cite{osipov2010a, osipov2010b}.
Some weighted extensions~\cite{kislyakov2008, borovitskiy2021weighted} also exist in the literature along with versions of the result for non-Lebesgue norms~\cite{malinnikova2019,ho2021,ho2022}.
A survey on some of these and other similar results may be found in the~review~\cite{lacey2003}.

As for the alternative Fourier transforms, we mention the paper~\cite{osipov2017} of Osipov, which addressed the version of inequality~\eqref{eqn:rdf_classic_dual} for the one-parameter Walsh--Fourier series, and the works of Tselishchev~\cite{tselishchev2021} and the author~\cite{borovitskiy2022littlewood}, which are dedicated to the one-parameter Vilenkin case and the two-parameter Walsh case respectively, and which we build upon and generalize in this work.\footnote{Going from the one-parameter Walsh setting \cite{osipov2017} the one-parameter Vilenkin setting \cite{tselishchev2021} turns out to be a much harder problem than it may seem, requiring a substantially different combinatorial argument. Furthermore, as in the case of multi-parameter singular integral operators, going from the one-parameter setting to the two-parameter setting \cite{borovitskiy2022littlewood} is involved, while going further to the arbitrary number of parameters is even more so (cf. e.g. \cite{fefferman1986}). In~doing that, this paper brings the considerations of papers \cite{osipov2017, tselishchev2021, borovitskiy2022littlewood} to a natural end.}
It is important to note that we only consider \emph{bounded} Vilenkin systems in this paper, with proofs heavily relying on the boundedness.
Questions related to \emph{unbounded} Vilenkin systems are also studied in the literature (see e.g. \cite{watari1958, young1976}) with techniques and results differing considerably.
We refer the reader to \cite{tselishchev2022}---where a version of the Rubio de Francia inequality has recently been proven for unbounded one-parameter Vilenkin systems---for a relevant discussion.

\paragraph{Proof concept} The scheme of the proof can be described as follows.
Generalizing the Tselishchev's geometric argument~\cite{tselishchev2021} to the multi-parameter case we are able to reduce the main theorem to the problem of boundedness of certain multi-parameter operators similar in spirit to the multi-parameter singular integrals.
By using the theory of multi-parameter martingale Hardy spaces and their atomic decomposition formulated by Weisz \cite{weisz2013}, we prove the multi-parameter analog of the Gundy's theorem on boundedness of operators taking martingales to measurable functions (Theorem~\ref{thm:gundy_Dd}).
This theorem, due to its simple assumptions, may be of independent interest outside the scope of this work.
Adapting our version of Gundy's theorem to the problem at hand (this yields a somewhat sharpened version of it---Corollary~\ref{thm:gundy_cor}) and identifying the integrable functions with martingales, we prove that the required operators are indeed bounded.

\paragraph{Outline} The paper's structure does not quite follow the proof concept described above.
First, in Section~\ref{sec:prelim} we define the notions required later in the paper and briefly describe the existing theory of multi-parameter martingale Hardy and Lebesgue spaces.
After this, in Section~\ref{sec:gundy} we prove a version of Gundy's theorem and its corollary that is useful for our particular scenario.
In Section~\ref{sec:operators} we use Gundy's theorem to prove boundedness of the two operators key to the main theorem.
Then, in Section~\ref{sec:partition} we recall the one-parameter partitioning argument of Tselishchev to subdivide an arbitrary interval into subintervals with favorable properties.
We use this partition argument and two lemmas from Section~\ref{sec:operators} to prove the main theorem in Section~\ref{sec:proof} by leveraging additional geometric and combinatorial considerations.
In Section~\ref{sec:extensions} we discuss some of the corollaries and the extensions of the main theorem, including a weaker version of the Rubio de Francia inequality for multi-parameter Vilenkin system for exponents $0 < p \leq 1$ and a version of the one-parameter Rubio de Francia inequality for exotic sets in place of the intervals.
We also show how to disprove the analogous inequality for partitions of $\Z_+^D$ into arbitrary sets. 

\paragraph{Notation}
By $\Z_+$ we mean $\N \cup \cbr{0}$ (we assume $0 \not\in \N$).
By rectangles we mean the arbitrary-dimensional product of intervals in $\Z_+$ themselves defined by $[a, b) = \Set{x \in \Z_+}{ a \leq x < b}$.
In cases when a martingale $f$ is of form $f_n = \E_n g$ for some measurable function $g$, we often identify $f$ and $g$ and use the same letter for both of them. 
Notation $a \lesssim b$ means that $a \leq C b$ for some positive constant $C$ that is not important for the discussion. 
For $n, m \in \Z_+^D$ we write $n \leq m$ ($n < m$) when $n_d \leq m_d$ ($n_d < m_d$) for all $1 \leq d \leq D$.
For $n \in \Z_+^D$ we write $n-\v{1}$ to denote the vector with components $\del{n-\v{1}}_d = \max(n_d - 1, 0)$, $d = 1, \ldots, D$.
We use the Fraktur font to distinguish Lebesgue and Hardy spaces of martingales from the respective spaces of functions, i.e. $\c{H}^p$ and $\c{L}^p$ are the martingale Hardy and Lebesgue spaces, while $H^p$ and $L^p$ are the usual Hardy and Lebesgue spaces of (equivalence classes of) functions\footnote{Strictly speaking, the elements of $H^p$ may be distributions, but we do not dwell on this.}.
The symbols $\c{H}^p$, $\c{L}^p$ (resp. $H^p$, $L^p$) correspond to the spaces of martingales (resp. functions) on $[0,1)$ or $[0,1)^D$, depending on the~context.

\section{Preliminaries} \label{sec:prelim}

Here we introduce the preliminary theory.
First, we discuss Vilenkin systems of one and many parameters.
Then, we define the corresponding martingale Lebesgue and Hardy spaces and present the theory that allows to establish boundedness of a certain class of operators taking martingales to measurable functions.
Further in the paper this preliminary theory will be used to prove a multi-parameter version of Gundy's theorem and, afterwards, to prove key lemmas for the main theorem.

Although, strictly speaking, the theory of $l^2$-valued functions and martingales will be needed to prove said lemmas, in this section we present only the scalar-valued results.
The reason for that is to simplify notation: while the theory generalizes to the $l^2$-valued case in a straightforward manner, the notation for the generalization ends up being too cumbersome.
We will discuss transferring the scalar-valued results to the $l^2$-valued case once more in Section~\ref{sec:prelim_l2_val}.

\subsection{Vilenkin system}
\label{sec:prelim_vilenkin}

Vilenkin systems are orthonormal bases in $L^2 [0,1)$ that generalize the classical Walsh system \cite{vilenkin1947, walsh1923}.
Similar to the Walsh system, Vilenkin systems correspond to some procedure of dividing the interval $[0,1)$ into subintervals, only instead of dividing each interval in two at every step, Vilenkin systems arise from dividing each interval into a variable number $p_n \geq 2$ of subintervals at $n$-th step.
To every sequence $\v{p} = \cbr{p_n}_{n \in \N}$, $p_n \geq 2$ there corresponds a different Vilenkin system.
The Vilenkin system corresponding to $p_n \equiv 2$ is the Walsh system.
Let us assume $\v{p}$ fixed and denote $P_n = p_1 \cdot \ldots \cdot p_n$, $P_0 = 1$ for convenience.
We will use terms \emph{Vilenkin function} and \emph{Walsh function} to refer to an element of some Vilenkin system or an element of the Walsh system respectively.

\paragraph{Mixed-radix numerical system} Apart from the procedure of dividing the interval $[0, 1)$ into subintervals, a sequence $\v{p}$ defines a mixed radix numerical system, meaning that each $x \in [0, 1)$ and each $n \in \Z_+$ have representation of form
\<
\label{eqn:vilenkin_decomposition_real}
x &= \sum_{k=1}^\infty x_k/P_k, &\t{~where} 0 \leq x_k < p_k \t{and} x_k \t{are integers,}
\\
\label{eqn:vilenkin_decomposition_integer}
n &= \sum_{k=1}^\infty n_k \cdot P_{k-1}, &\t{~where} 0 \leq n_k < p_k \t{and} n_k \t{are integers.}
\>
We will sometimes call $x_k$ (resp. $n_k$) in the expressions above the \emph{digits} of $x$ (resp. $n$) in the mixed radix system $\v{p}$.
Note that representation~\eqref{eqn:vilenkin_decomposition_real} is not unique.\footnote{This representation is not unique in the exact same way as the representation of a real number as an infinite decimal fraction is not unique.
In fact, the decimal representation is the special case of~\eqref{eqn:vilenkin_decomposition_real} with $p_n \equiv 10$.
}
This is easy to fix though: if there are two representations of form~\eqref{eqn:vilenkin_decomposition_real}, choosing the one with $x_k \to 0$ makes the representation unique \cite{weisz2013}.

\paragraph{Vilenkin functions} After this preparation, let us define the generalized Ra\-de\-ma\-cher functions, the main building blocks of the Vilenkin functions.
\begin{definition}
The $n$-th \emph{generalized Rademacher function} $r_n: [0, 1) \to \C$ is
\[ \label{eqn:rademacher_definition}
r_n(x) = \exp(2 \pi i x_n / p_n),
\]
where $x_n$ are digits of $x$ in the mixed radix system $\v{p}$ and $n \in \N$.
\end{definition}

And now we are ready to define the Vilenkin functions.

\begin{definition}
The $n$-th \emph{Vilenkin function} $v_n: [0,1) \to \C$ is defined by
\[ \label{eqn:vilenkin_definition}
v_n (x) = \prod_{k=1}^\infty r_k(x)^{n_k},
\]
where $n_k$ are digits of $n$ in the mixed radix system $\v{p}$ and $n \in \Z_+$.
\end{definition}
The Vilenkin system $\cbr{v_n}_{n \in \Z_+}$ is an orthonormal basis of the space $L^2 [0,1)$~\cite{vilenkin1947}.

\paragraph{Group structure} Vilenkin systems also arise as characters of certain topological groups called \emph{Vilenkin groups}.
Let us explore this characterization formally.
\begin{definition}
The set $G_{\v{p}}$ of sequences $\cbr{x_k}_{k \in \N}$ such that $x_k \in \Z_+$ and ${0 \leq x_k < p_k}$ together with operation $\bullet_{\v{p}}: G_{\v{p}} \x G_{\v{p}} \to G_{\v{p}}$ of bit-wise addition modulo $\v{p}$ defined formally by
\[
\cbr{x \bullet_{\v{p}} y}_k = x_k + y_k \bmod p_k
\]
is called the \emph{Vilenkin group}.
\end{definition}
\begin{remark}
The above is trivially equivalent to defining Vilenkin group as the infinite direct product of cyclic groups
\[
G_{\v{p}} = \prod_{k=1}^\infty \del{\Z/p_k \Z}
.
\]
Assuming the discrete topology for cyclic groups, we can endow the Vilenkin group with the standard product topology.
\end{remark}
Through decomposition~\eqref{eqn:vilenkin_decomposition_real} the elements of $G_{\v{p}}$ can be identified with numbers from $[0,1)$.
Moreover, this can be done by a measure preserving transform, which maps the Haar measure of $G_{\v{p}}$ into the Lebesgue measure over $[0,1)$ \cite{vilenkin1947}.

The characters of the group $G_{\v{p}}$, if its elements are thought of as functions over $[0,1)$, are exactly the Vilenkin functions.
Furthermore, since characters of a group themselves form a group, there is a natural group structure for Vilenkin system.
It is given by identifying each Vilenkin function $v_n$ with a non-negative integer $n \in \Z_+$ and considering the operation $\dot{+}: \Z_+ \x \Z_+ \to \Z_+$ such that
\[
\del[2]{\sum_{k=1}^\infty \alpha_k P_{k-1}}
\dot{+}
\del[2]{\sum_{k=1}^\infty \beta_k P_{k-1}}
=
\sum_{k=1}^\infty (\alpha_k + \beta_k \bmod p_k) P_{k-1}
.
\]
This means that for $n, m \in \Z_+$, we have $v_n (x) v_m (x) = v_{n \dot{+} m} (x)$.
As a group then, Vilenkin system is isomorphic to the direct sum (not product) of cyclic groups $\bigoplus_{k=1}^\infty \del{\Z/p_k \Z}$.
See \cite{vilenkin1947, weisz2013} and references therein for the proofs and comments on the claims made here.

\subsection{Multi-parameter Vilenkin system}
\label{sec:prelim_multiparam_vilenkin}

Let $D \in N$.
A $D$-parameter Vilenkin system is a product of $D$ (different) Vilenkin systems.
\begin{definition}
Consider $D$ Vilenkin systems $\cbr[1]{v^{(d)}_n}_{n \in \Z_+}$, $d = 1, \ldots, D$ corresponding to (possibly different) sequences $\v{p}_d$.
Then a $D$-parameter family of functions~$\cbr{v_n}_{n \in \Z_+^D}$, where $v_n: [0,1)^D \to \C$ are defined by
\[
v_n = v^{(1)}_{n_1} \otimes \ldots \otimes v^{(D)}_{n_D}
&&
\text{i.e.}
&&
v_n (x) = v^{(1)}_{n_1}(x_1) \cdot \ldots \cdot v^{(D)}_{n_D}(x_D),
\]
is called a \emph{$D$-parameter Vilenkin system}.
\end{definition}

A $D$-parameter Vilenkin system is a group with the product group structure.
That is, if we identify the $D$-parameter Vilenkin functions with elements of $\Z_+^D$, the group operation for some elements $n, m \in \Z_+^D$ is given by
\[
n \dot{+} m = (n_1 \dot{+} m_1, n_2 \dot{+} m_2, \ldots, n_D \dot{+} m_D)
.
\]
With this we have $v_n (x) v_m(x) = v_{n \dot{+} m}(x)$ as in the one-parameter case.

\paragraph{Vilenkin--Fourier series} A multi-parameter Vilenkin system is always an orthonormal basis of the space $L^2 [0,1)^D$.
Hence, any square-integrable function~$g$ may be represented as a sum of the \emph{Vilenkin-Fourier series}
\[
g(x) = {\sum}_{l \in \Z_+^D} c_l \cdot v_l(x)
\qquad
\t{with coefficients} c_l = \innerprod{g}{v_l}_{L^2 [0,1)^D}.
\]
These $c_l$ can be called \emph{the Vilenkin--Fourier coefficients}, while the set of indices ${\spec g = \Set{l \in \Z_+^D}{c_l \not= 0}}$ can be called \emph{the Vilenkin--Fourier spectrum}.

\paragraph{Isomorphism type} A $D$-parameter Vilenkin system is always group-isomorphic to some one-parameter Vilenkin system, for instance to one with
\[
\v{p} = (p^{(1)}_1, .., p^{(D)}_1, p^{(1)}_2, ..,  p^{(D)}_2, .. ).
\]
This does not mean that every result for $D$-parameter Vilenkin systems follows trivially from the analogous result for one-parameter systems, and the Rubio de Francia inequality is an example of this.
On contrary, this will allow us to deduce some curious one-parameter results from the $D$-parameter one.
We study this connection in more detail in Section~\ref{sec:extensions}.

\subsection{Multi-parameter Vilenkin martingales}
\label{sec:prelim_vilenkin_martingales}

A combination of sequences $\v{p}_1, \ldots, \v{p}_D$ defines a \emph{generalized filtration} that is deeply connected with the $D$-parameter Vilenkin system.
\begin{definition}
The \emph{$D$-parameter Vilenkin filtration} corresponding to a combination of sequences $\v{p}_1, \ldots, \v{p}_D$ is the family $\cbr{\c{F}_{n}}_{n \in \Z_+^D}$ of $\sigma$-algebras
\[
\c{F}_{n}
=
\sigma
\del{
    \cbr{
        \srbr{\frac{l_1}{P^{(1)}_{n_1}}, \frac{l_1+1}{P^{(1)}_{n_1}}}
        \x
        \ldots
        \x
        \srbr{\frac{l_D}{P^{(D)}_{n_D}}, \frac{l_D+1}{P^{(D)}_{n_D}}}
        :
        0 \leq l_d < P^{(d)}_{n_d}
    }
}
,
\]
where $\sigma(\c{H})$ denotes the $\sigma$-algebra generated by elements of set $\c{H}$.
\end{definition}
We call $\srbr[1]{l_d/P^{(d)}_{n_d}, (l_d+1)/P^{(d)}_{n_d}} \subseteq [0,1)$ \emph{Vilenkin intervals} and their products contained in $[0,1)^D$ \emph{Vilenkin rectangles} by analogy with the dyadic intervals and the dyadic rectangles.

Define operator $\E_{n}$~to be the conditional expectation with respect to $\c{F}_{n}$.
For $n, m \in \Z_+^D$ we write $n \leq m$ when $n_d \leq m_d$ for all $d = 1, \ldots, D$.
With this, we are able to define $D$-parameter Vilenkin martingales.

\begin{definition}
A family $f = \cbr{f_{n}}_{n \in \Z_+^D}$ of integrable functions over $[0,1)^D$ is a \emph{$D$-parameter Vilenkin martingale} if
\1 $f_n$ is measurable with respect to $\sigma$-algebra $\c{F}_{n}$ for every $n \in \Z_+^D$,
\2 $\E_n f_m = f_n$ for every $n, m$ such that $n \leq m$.
\0
\end{definition}
Any integrable function $g: [0, 1)^D \to \C$ defines a martingale $\cbr{\E_n g}_{n \in \Z_+^D}$.

\begin{samepage}
\paragraph{Martingale differences} To define the martingale differences with respect to a generalized filtration (as above), let us note that the operator $\E_n$ can be represented as the composition
$
\E_{n} = \E_{n_1}^1 \E_{n_2}^2 \ldots \E_{n_D}^D
$
of its one-parameter counterparts $\E_{n_d}^d$.
Here $\E_{n_d}^d$ can be formally defined as the projection onto the $\sigma$-algebra $\c{F}^d_{n_d} = \cup_{k \in \Z_+^D, k_d = n_d} \c{F}_k$.
The martingale difference $\Delta_n$ operator acting on an integrable function $g$ is defined to be
\[
\Delta_{n} g
=
\Delta_{n_1}^1
\Delta_{n_2}^2
\ldots
\Delta_{n_D}^D
g
, \quad \text{where} \quad
\Delta_{n_d}^d = {\E}_{n_d}^d - {\E}_{n_d-1}^d
\]
and where for $n_d-1 = -1$ we assume ${\E}_{n_d-1}^d = 0$.
\end{samepage}
If we expand the expression above, we also get the representation
\[ \label{eqn:difference_definition}
\Delta_{n} g = {\sum}_{\varepsilon_1, \ldots, 
\varepsilon_D \in \cbr{0, 1}} (-1)^{\varepsilon_1 + \ldots + \varepsilon_D} {\E}_{n_1 - \varepsilon_1, \ldots, n_D - \varepsilon_D} g
,
\]
where, similarly, the terms containing index $n_d - 1 = -1$ are assumed to be zero.
For instance, for $d = 2$ we get
\[
\begin{aligned} 
\Delta_{n} g
:&=
({\E}_{n_1}^1 - {\E}_{n_1-1}^1)
({\E}_{n_2}^2 - {\E}_{n_2-1}^2)
g
\\
&=
{\E}_{n_1, n_2} g - {\E}_{n_1 - 1, n_2} g - {\E}_{n_1, n_2 - 1} g + {\E}_{n_1 - 1, n_2 - 1} g
.
\end{aligned}
\]
Operators $\E_n$ and $\Delta_n$ can be extended to act on martingales by simply putting $\E_{n} g = g_n$ and treating~\eqref{eqn:difference_definition} as the definition of $\Delta_n g$.
Obviously, we have $\E_n g = \sum_{m \leq n} \Delta_m g$ both for functions and martingales.

\paragraph{Connection to the Vilenkin system} The fundamental connection between the $D$-parameter Vilenkin martingales and the $D$-parameter Vilenkin system is given by the following two equations \cite{weisz2013}:
\<
\label{eqn:expectation_vilenkin_sum}
\del[1]{{\E}_{n} f} (x)
&=
\sum\limits_{l_1 = 0}^{P^{(1)}_{n_1} - 1}
    ...
    \sum\limits_{l_D = 0}^{P^{(D)}_{n_D} - 1}
        \innerprod{f_m}{v_l(\cdot)}
        \ 
        v_l(x)
&
&
\t{with}
m \geq n,
\\
\label{eqn:difference_vilenkin_sum}
\del{\Delta_{n} f} (x)
&=
\sum\limits_{l \in \delta_{n}}
        \innerprod{f_m}{v_l(\cdot)}
        \ 
        v_l(x)
&
&
\t{with}
m \geq n,
\>
where $f$ is a $D$-parameter Vilenkin martingale, $\innerprod{\cdot}{\cdot}$ is the inner product in the space $L^2 [0,1)^D$, and
$
\delta_{n}
=
[P^{(1)}_{n_1-1}, P^{(1)}_{n_1})
\x
..
\x
[P^{(D)}_{n_D-1}, P^{(D)}_{n_D})
\subseteq
\Z_+^D
,
$
with~$P^{(d)}_{-1} = 0$.

\subsection{Lebesgue and Hardy spaces of Vilenkin martingales}
\label{sec:prelim_spaces}

Here we define the Lebesgue and Hardy spaces of $D$-parameter Vilenkin martingales and discuss their very basic properties.

\paragraph{Bounded Vilenkin systems} Throughout this and the following sections we consider only the \emph{bounded Vilenkin systems}: the ones that correspond to such sequences $\v{p} = \cbr{p_n}_{n \in \N}$ that $p_n \leq C_{\v{p}}$ for some global $C_{\v{p}} > 0$.
Bounded Vilenkin systems correspond to \emph{regular} (generalized) filtrations, for which there exists a global $C > 0$ such that for any non-negative martingale $f$ the relation $f_n \leq C f_{n-1}$ holds in one-parameter case or relations
\[
\begin{aligned}
f_{n_1, n_2, \ldots, n_D} &\leq C f_{n_1 - 1, n_2, \ldots, n_D},
\\
f_{n_1, n_2, \ldots, n_D} &\leq C f_{n_1, n_2 - 1, \ldots, n_D},
\\
&\dots
\\
f_{n_1, n_2, \ldots, n_D} &\leq C f_{n_1, n_2, \ldots, n_D - 1}
\end{aligned}
\]
hold in the $D$-parameter case.
The regularity is an important assumption for a filtration used in the atomic theory of Hardy spaces, that we briefly describe in Section~\ref{sec:prelim_boundedness}.

\paragraph{Lebesgue spaces} We say that a martingale $f$ lies in the Lebesgue space $\c{L}^p$ and write $f \in \c{L}^p$, for some $0 < p \leq \infty$, if $f_n \in L^p$ for all $n \in \Z_+^D$ and
\[
\norm{f}_{\c{L}^p} = \sup\limits_{n \in \Z_+^D} \norm{f_n}_{L^p} < \infty.
\]
As we have mentioned before, to every integrable function there corresponds a martingale.
For a martingale $f \in \c{L}^p$ for $1 < p < \infty$ the converse is also true: there exists a function $g \in L^p$ such that $f_n = \E_n g$ and
\[
\lim\limits_{\min(n_1, \ldots, n_D) \to \infty} \norm{f_n - g}_{L^p} = 0
,
\qquad
\qquad
\qquad
\norm{f}_{\c{L}^p} = \norm{g}_{L^p}
.
\]
Following the common practice, we will often identify martingale $f$ with the function $g$ and refer to $g$ by the same symbol $f$, treating it as a function or a martingale depending on the particular context.

\paragraph{Hardy spaces} There exists a number of different definitions of martingale Hardy spaces.
In the most general setting, they may result in different objects, but for regular filtrations all of them are equivalent \cite{weisz2006, weisz2013}.
We thus use the definition most convenient for us.
First, we introduce the \emph{martingale Littlewood--Paley square function} $S$ that is given by
\[
    S(f) = \del[2]{{\sum}_{n \in \Z_+^D} \abs{\Delta_n f}^2}^{1/2}.
\]
It defines the martingale Hardy spaces $\c{H}^p$, $0 < p \leq \infty$, the spaces of multiparameter martingales $f$ such that $\norm{f}_{\c{H}^p} = \norm{S f}_{L^p} < \infty$.

For $1 < p < \infty$ we have (cf. \cite{weisz2013}) that $\norm{S(u)}_{L^p} \sim \norm{u}_{\c{L}^p}$.
Thus, for ${1 < p < \infty}$, the Lebesgue space $\c{L}^p$ coincides with the Hardy space $\c{H}^p$.
This does not hold for $p \leq 1$.
Hardy spaces serve, in the context of space interpolation, as the natural extension of the Lebesgue scale for exponents $p \leq 1$.
In particular, one can prove that if an operator $T$ is bounded from $\c{H}^{p_0}$ to $L^{p_0}$ and from $\c{L}^{p_1}$ to $L^{p_1}$ for some $p_0 \leq 1 < p_1$, then $T$ is also bounded from $\c{L}^{p}$ to $L^{p}$ for $1 < p \leq p_1$.

This provides an efficient mean of proving boundedness of an operator on $\c{L}^p$ for all $1 < p \leq 2$ simultaneously, by reducing it to the problem of boundedness from $\c{L}^2$ to $L^2$, which is relatively simple, and to the problem of boundedness from $\c{H}^{p_0}$ to $L^{p_0}$ for some $p_0 \leq 1$.

Since establishing boundedness of the certain operator for all $1 < p \leq 2$ will be central for the proof of the main theorem, and since the tools for checking boundedness from $\c{L}^2$ to $L^2$ are classical and fairly simple, we proceed to describe the machinery that helps proving boundedness of operators from $\c{H}^{p_0}$ to $L^{p_0}$.

\subsection{Atomic decomposition and operator boundedness}
\label{sec:prelim_boundedness}

Here we discuss the atomic decomposition of the multi-parameter martingale Hardy spaces and a result that proves boundedness of a certain class of operators which map martingales from Hardy spaces to functions from Lebesgue spaces.

\paragraph{Atomic decomposition} Elements of Hardy spaces can be represented as weighted sums of simple functions called \emph{atoms}.
We give the precise definition of \emph{atoms} later, as it differs depending on the number of parameters or the kind of Hardy spaces being used (martingale/non-martingale).

\begin{theorem}[Atomic decomposition]
A martingale $f$ is in $\c{H}^p$ $(0 < p \leq 1)$ if and only if there exists a sequence $a_m$ of relatively simple martingales called $p$-atoms (to be defined in the future) and a sequence $\mu_m$ of real numbers such that $f = \sum_{m=0}^{\infty} \mu_m a_m$, formally meaning that
\<
& \sum_{m=0}^{\infty} \mu_m {\E}_n a_m \, \stackrel{a.e.}{=} \, {\E}_n f \t{ for all } n,
&
& \sum_{m=0}^{\infty} \abs{\mu_m}^p < \infty.
\>
Moreover, $\norm{f}_{\c{H}^p} \sim \inf \del{\sum_{m=0}^{\infty} \abs{\mu_m}^p}^{1/p}$, where the infimum is taken over all representations of martingale $f$ of the above form.
\end{theorem}
\begin{proof}
Theorems 1.14 and 1.16 of \cite{weisz2013}.
\end{proof}
Regardless of what atoms actually are, we can easily show that the behavior of operator on them alone determines its boundedness in Hardy spaces.
\begin{proposition}
Let $B$ be some Banach space.
A linear operator $T: \c{H}^p \to B$ is bounded if and only if for any atom $a$, we have $\norm{T a}_{B} \leq C$ for some constant~$C$ that is independent of $a$.
\end{proposition}
\begin{proof}
Assume first that $\norm{T a}_{B} \leq C$ for every atom $a$.
Consider the atomic decomposition $f = \sum_{m=0}^{\infty} \mu_m a_m$. Since $p \leq 1$, we have $\del{a + b}^p \leq a^p + b^p$, thus
\[
\norm{T f}_{B}^p
\leq
\sum_{m=0}^{\infty} \abs{\mu_m}^p \norm{T a_m}_{B}^p
\leq 
C^p \sum_{m=0}^{\infty} \abs{\mu_m}^p.
\]
By taking decompositions of $f$ with $\del{\sum_{m=0}^{\infty} \abs{\mu_m}^p}^{1/p}$ close to the infimum over all such representations, we see that $\norm{T f}_{B} \leq C \norm{f}_{\c{H}^p}$.

Now assume that $T$ is bounded.
Since atoms lie in $\c{H}^p$, for any atom $a$ we have $\norm{T a}_{B} \lesssim \norm{a}_{\c{H}^p}$.
From the relation $\norm{f}_{\c{H}^p} \sim \inf \del{\sum_{m=0}^{\infty} \abs{\mu_m}^p}^{1/p}$ it is obvious that the norm of all atoms is bounded.
Hence the converse.
\end{proof}
Intuitively, atoms form the core of the Hardy spaces.
Thus it should come as not surprise that with Hardy spaces getting more complex (with the number of parameters $D$ getting larger), the atoms become more complex as well.
For the one-parameter case, atoms are given by the following definition.
\begin{definition}
A function $a \in L^2 [0,1)$ (regarded as a martingale) is a one-parameter $p$-atom if the following holds:
\begin{enumerate}
    \item $\supp a \subseteq I$ for some Vilenkin interval $I \subseteq [0, 1)$,
    \item $\norm{a}_{L^2} \leq \abs{I}^{1/2-1/p}$,
    \item $\int_{[0,1)} a(x) \d x = 0$.
\end{enumerate}
\end{definition}
For two and more parameters, the atoms are much more complex objects.
To define them we need the notion of a \emph{maximal Vilenkin rectangle}.
If $F \subseteq [0, 1)^D$ is open, then a Vilenkin rectangle $R \subseteq F$ is called maximal if and only if $R \subseteq R' \subseteq F$ implies $R' = R$ for any Vilenkin rectangle $R'$.
The set of all maximal Vilenkin rectangles is denoted by $\c{M}(F)$.
\begin{definition}
A function $a \in L^2 [0,1)^D$ (regarded as a martingale) is a $D$-parameter $p$-atom for $D \geq 2$ if the following holds:
\begin{enumerate}
    \item $\supp a \subseteq F$ for some open set $F \subseteq [0, 1)^D$,
    \item $\norm{a}_{L^2} \leq \abs{F}^{1/2-1/p}$,
    \item $a$ can be further decomposed into the sum $a = \sum_{R \in \c{M}(F)} a_R$ of simpler functions $a_R \in L^2$ called \emph{elementary particles} such that
    \begin{enumerate}
        \item $\supp a_R \subseteq R \subseteq F$, where $R \in \c{M}(F)$ is a maximal Vilenkin rectangle,
        \item for all $1 \leq d \leq D$ and arbitrary $x_1, .., x_{d-1}, x_{d+1}, .., x_D \in [0,1)$
        \[
            \int_{[0,1)} a_R(x_1, .., x_d, .., x_D) \d x_d = 0
            &&
            \text{i.e.}
            &&
            {\E}_{0}^d a_R = 0
            ,
        \]
        \item for every disjoint partition $\cbr{\c{P}_l}$ of $\c{M}(F)$,
        \[
            \del[2]{\sum_{l} \norm[1]{\sum_{R \in \c{P}_l} a_R}_{L^2}^2}^{1/2} \leq \abs{F}^{1/2-1/p}
            .
        \]
    \end{enumerate}
\end{enumerate}
\end{definition}

While in the one-parameter case establishing boundedness of operators by checking $\norm{T a} \leq C$ for all atoms is a reasonable approach, in multi-parameter case this approach is hindered by the complexity of atoms.
It turns out though that there exist simpler sufficient conditions for boundedness of operators.
These are motivated by one of the most basic results of classical singular integrals theory, namely the following.
\begin{theorem}
    If $T: L^2(\R) \to L^2(\R)$ is a bounded linear operator and for any function $a$ supported on an interval $I$ satisfying $\int_{\R} a(x) \d x = 0$ we have
    \[
        \int_{\R \setminus I^{(r)}} \abs{T a}(x) \d x \lesssim \norm{a}
        ,
    \]
    then $T$ is bounded from $L^p$ to $L^p$ for $1 < p \leq 2$.
    Here $I^{(r)}$ is the interval of length $2^r \abs{I}$ that has the same center as $I$.
\end{theorem}
This theorem is a simple consequence of the atomic decomposition for the one-parameter case.
R. Fefferman \cite{fefferman1986} found out that a close analog of this statement works in two-parameter trigonometric case, while Weisz \cite{weisz1997, weisz2013} adapted his ideas to show the same in the martingale case.

Prior to formulating the theorem for $2$-parameter Vilenkin martingales, we define the analog of $I^{(r)}$ for Vilenkin intervals and rectangles.
For a Vilenkin interval $I \subseteq [0,1)$ we define $I^{(1)}$ to be the Vilenkin interval such that $I \subseteq I^{(1)}$, $I \not= I^{(1)}$ and such that any other Vilenkin interval with these properties contains the interval $I^{(1)}$ (the left ends of $I$ and $I^{(1)}$ coincide). We define $I^{(r)}$ to be the Vilenkin interval obtained by applying this procedure $r$ times.
For a Vilenkin rectangle $R = R^1 \x \ldots \x R^D$ we put $R^{(r)} = \del[0]{R^1}^{(r)} \x \ldots \x \del[0]{R^D}^{(r)}$.
\begin{theorem} \label{thm:quasi_local_2d}
If $T: \c{L}^2 [0,1)^2 \to L^2 [0,1)^2$ is a bounded linear operator and for some $p_0 \leq 1$ there exists $\eta > 0$ such that for any function $a$ supported on a Vilenkin rectangle $R$ and satisfying $\norm{a}_{L^2} \leq \abs{R}^{1/2-1/p_0}$, $\int_0^1 a(x_d) \d x_d = 0$ for $d = 1, 2$ we have for any $r \in \N$
\[ \label{eqn:quasi_local_2d}
    \int_{[0,1)^2 \setminus R^{(r)}} \abs{T a}^{p_0}(x) \d x \leq C 2^{\eta r}
    ,
\]
then $T: \c{H}^{p} \to L^{p}$ is bounded for $p_0 \leq p \leq 2$. Furthermore, since $\c{H}^p$ and $\c{L}^p$ coincide for $1 < p \leq 2$, we have $T: \c{L}^p \to L^p$ bounded in this case.
\end{theorem}
\begin{proof}
Theorem 1.41 of \cite{weisz2013} contains the proof for the Walsh system, the proof for bounded Vilenkin systems is the same.
The formal statement for bounded Vilenkin systems can be found in \cite[Theorem 13]{weisz2004}, though without an explicit proof.
\end{proof}
The proof for the two-parameter case is based upon the ingenious Journe's covering lemma.
For $D \geq 3$, the corresponding statement is no longer true \cite{weisz2013}. Instead, a more complicated criterion exists, based on a version of Journe's covering lemma formulated by Pipher \cite{pipher1986}.
Before stating the assertion, let us consider a reformulation of the $2$-parameter criterion.
It is obvious that if $R = I \x J$, then instead of~\eqref{eqn:quasi_local_2d} we can ask (denoting $(I^{(r)})^c = [0,1) \setminus I^{(r)}$ and analogously $(J^{(r)})^c = [0,1) \setminus J^{(r)}$) for
\[
    \int_{(I^{(r)})^c \x [0, 1)} \abs{T a}^{p_0}(x) \d x \leq C 2^{\eta r}
    \t{~and~}
    \int_{[0, 1) \x (J^r)^c} \abs{T a}^{p_0}(x) \d x \leq C 2^{\eta r}
    .
\]
Or even, we can ask
\[
    \begin{aligned} \label{eqn:quasi_local_2d_reformulated}
    \int_{(I^{(r)})^c \x J} \abs{T a}^{p_0}(x) \d x &\leq C 2^{\eta r},
    &
    \int_{I \x (J^{(r)})^c} \abs{T a}^{p_0}(x) \d x &\leq C 2^{\eta r},
    \\
    \int_{(I^{(r)})^c \x J^c} \abs{T a}^{p_0}(x) \d x &\leq C 2^{\eta r},
    &
    \int_{I^c \x (J^{(r)})^c} \abs{T a}^{p_0}(x) \d x &\leq C 2^{\eta r}.
    \end{aligned}
\]
The general multi-parameter criterion will generalize this form of condition~\eqref{eqn:quasi_local_2d}.
Let us present the formal statement now.
\begin{theorem} \label{thm:quasi_local_bounded}
Let $D \geq 2$ and $0 < p_0 < 2$.
If $T: \c{L}^2 [0,1)^D \to L^2 [0,1)^D$ is a bounded linear operator and there exist $\eta_1, \ldots, \eta_D > 0$ such that for every simple $p_0$-atom $a$ (to be defined) the following holds.
If $a$ is (up to a permutation of coordinates) supported on $I_1 \x \ldots \x I_j \x A$, $j < D$, where $I_i \subseteq [0,1)$ are Vilenkin intervals and $A \subseteq [0,1)^{D-j}$ is a measurable set, then for every $r_1, \ldots, r_j \in \N$
\[ \label{eqn:quasi_local_Dd_1}
\int_{\del[1]{I_1^{(r_1)}}^c \x \ldots \x \del[1]{I_j^{(r_j)}}^c \x A } \abs{T a}^{p_0} (x) \d x \leq C 2^{-\eta_1 r_1} \cdot \ldots \cdot 2^{-\eta_j r_j}.
\]
If $j = D - 1$ and $A = I_D$ is a Vilenkin interval, then
\[ \label{eqn:quasi_local_Dd_2}
\int_{\del[1]{I_1^{(r_1)}}^c \x \ldots \x \del[1]{I_{D-1}^{(r_{D-1})}}^c \x I_D^c} \abs{T a}^{p_0} (x) \d x \leq C 2^{-\eta_1 r_1} \cdot \ldots \cdot 2^{-\eta_{D-1} r_{D-1}}
\]
should hold as well.
Then $T: \c{H}^{p} \to L^{p}$ is bounded for $p_0 \leq p \leq 2$.
Furthermore, since $\c{H}^p$ and $\c{L}^p$ coincide for $1 < p \leq 2$, $T: \c{L}^p \to L^p$ is bounded in this case.
\end{theorem}

\begin{proof}
Theorem 1.45 of \cite{weisz2013} contains the proof for the Walsh system and $D=3$, the proof for bounded Vilenkin systems with $D \geq 3$ is the same.
The formal statement for bounded Vilenkin systems with $D \geq 3$ can be found in \cite[Theorem 14]{weisz2004}, though without an explicit proof.

Clearly, if $D = 2$, inequalities \eqref{eqn:quasi_local_Dd_1} and \eqref{eqn:quasi_local_Dd_2} instantiated at every permutation of coordinates give exactly inequalities~\eqref{eqn:quasi_local_2d_reformulated}.
Thus the case $D = 2$ follows from Theorem~\ref{thm:quasi_local_2d} and it is fair to say that this proposition generalizes the two-parameter statement.
\end{proof}

The simple atoms featured in Theorem~\ref{thm:quasi_local_bounded} are defined as follows.
\begin{definition}
A function $a \in L^2 [0,1)^D$, $D \geq 2$ (regarded as a martingale) is called a \emph{simple $p$-atom} if there exist Vilenkin intervals $I_i \subseteq [0,1)$, $i = 1, \ldots, j$ for some $1 \leq j \leq D-1$, such that
\begin{itemize}
\item $\supp a \subseteq I_1 \x .. \x I_j \x A$ for some measurable set $A \subseteq [0,1)^{D-j}$,
\item $\norm{a}_{L^2} \leq \del{\abs{I_1} \cdot \ldots \cdot \abs{I_j} \abs{A}}^{1/2-1/p}$,
\item $\int_{I_i} a(x_1, .., x_i, .., x_D) \d x_i = \int_A a(x_1, .., x_j, x_{j+1}, .., x_D) \d x_{j+1} \ldots \d x_{D} = 0$ i.e. $\E_0^i a = 0$ and $\E_0^{j+1} \!\!\!\!\!\ldots \E_0^{D} a = 0$ for all $i = 1, \ldots, j$.
\end{itemize}
A function for which the above holds after a permutation of coordinates is called a simple $p$-atom as well.
\end{definition}

The operators satisfying conditions of Theorem~\ref{thm:quasi_local_bounded} are called \emph{$H^{p_0}$-quasi-local}.

\subsection{The $l^2$-valued case}
\label{sec:prelim_l2_val}

When functions and martingales are $l^2$-valued, we simply change $\abs{\cdot}$ into $\norm{\cdot}_{l^2}$ in all definitions and statements.
As is often the case, the proofs do not change.
This way we may define the spaces $L^p(l^2)$, $\c{L}^p(l^2)$ and spaces $H^p(l^2)$, $\c{H}^p(l^2)$.
For instance, $\c{H}^p(l^2_{\Z})$ consists of $l^2_{\Z}$-valued martingales $f = \cbr{f_{n, l}}_{n \in \Z_+^D, l \in \Z}$ with
\[
\norm{f}_{\c{H}^p(l^2_{\Z})}
=
\norm{
S(f)
}_{L^p}
=
\norm{
\del[2]{{\sum}_{n \in \Z_+^D} \norm{\Delta_n f}_{l^2_{\Z}}^2}^{1/2}
}_{L^p}
<
\infty.
\]
Equation~\eqref{eqn:quasi_local_Dd_1} for operators $T: \c{L}^2(l^2) \to L^2$ does not change (but the simple atoms become $l^2$-valued) and for operators $T: \c{L}^2 \to L^2(l^2)$ it turns into
\[
\int_{\del[1]{I_1^{(r_1)}}^c \x \ldots \x \del[1]{I_j^{(r_j)}}^c \x A } \norm{T a}_{l^2}^{p_0} (x) \d x \leq C_{p_0} 2^{-\eta_1 r_1} \cdot \ldots \cdot 2^{-\eta_j r_j}.
\]
Equation~\eqref{eqn:quasi_local_Dd_2} is transformed analogously.
In this manner, every notion and result of Sections \ref{sec:prelim_multiparam_vilenkin}--\ref{sec:prelim_boundedness} can be transferred to the $l^2$-valued case, which we will use in the following.

\section{Multi-parameter Gundy theorem for bounded Vilenkin martingales}
\label{sec:gundy}

Theorem~\ref{thm:quasi_local_bounded} from the previous section allows us to prove, rather easily, a multi-parameter version of Gundy's theorem on boundedness of operators mapping martingales into measurable functions \cite{gundy1968, gundy1980}, or rather its more recent reformulation by Kislyakov \cite{kislyakov1987}.
This theorem gives a very simple to use sufficient conditions for boundedness of operators, and hence it is of interest to us and may be of independent interest as an additional result of this work.

\begin{theorem} \label{thm:gundy_Dd}
Fix $D \geq 2$ and consider a linear operator $V$ taking $D$-parameter Vilenkin martingales into measurable functions. Assume the following holds.
\1 $V: \c{L}^2 \to L^2$ is bounded.
\2 For any $D$-parameter martingale $f$ such that $f_{0} = 0$ and
\[
\label{eqn:gundy_f_cond}
\Delta_n f = \1_{e_n} \Delta_n f, \t{where} e_n \in \c{F}_{n-\v{1}} \t{with} \del{n-\v{1}}_d = \max(n_d - 1, 0)
,
\]
we have $\cbr{\abs{V f} > 0} \subseteq \bigcup_{n \in \Z_+^D \setminus \{0\}} e_n$.
\0
Then $V: \c{H}^{p} \to L^{p}$ is bounded for any $p \leq 1$ and $V: \c{L}^r \to L^r$ is bounded for~$1 < r \leq 2$.
\end{theorem}
\begin{proof}
We will show that $V$ is $H^p$-quasi-local for any $p \leq 1$ by checking conditions~\eqref{eqn:quasi_local_Dd_1} and~\eqref{eqn:quasi_local_Dd_2} for all simple atoms, the claim will then follow from Theorem~\ref{thm:quasi_local_bounded}.
Instead of carrying an arbitrary permutation of coordinates (as by Theorem~\ref{thm:quasi_local_bounded}) through the notation, we assume the most convenient one and rely on the symmetry of the statement we prove.

Consider a simple atom $a$ supported on a set $I_1 \x .. \x I_j \x A$, where $I_d$ are Vilenkin intervals and $A$ is a measurable set.
Denote $S = I_1 \x .. \x I_j \x [0,1)^{D-j}$ and find such index $N \in \Z_+^D$ that $S$ is an atom of $\c{F}_N$.
Such $N$ does exist because $S$ is a product of Vilenkin intervals, and it is exactly the minimal $N$ such that $S \in \c{F}_N$.
Note also that $N_{j+1} = .. = N_D = 0$.
We claim the following.
\<
\label{eqn:gundy_to_prove_1}
\Delta_{n} a &= \1_{I_1 \x .. \x I_j \x [0,1)^{D-j}} \Delta_{n} a,
&\t{if} n_d > N_d \t{for all} 1 \leq d \leq j,
\\
\label{eqn:gundy_to_prove_2}
\Delta_{n} a &= \1_{\emptyset} \Delta_{n} a,
&\t{otherwise.}
\>
Let us postpone the verification of these relations till the end of the proof and assume, for a time being, that they hold true.
From Equations \eqref{eqn:gundy_to_prove_1}, \eqref{eqn:gundy_to_prove_2} we have that $\Delta_n f = \1_{e_n} \Delta_n f$ for $e_n = I_1 \x .. \x I_j \x [0,1)^{D-j}$ or $e_n = \emptyset$ depending on the value of $n$.
The former case corresponds to such $n$ that $n_d > N_d ~\forall 1 \leq d \leq j$.
Here, since $e_n \in \c{F}_N$ while $N \leq n - \v{1}$, we have that $e_n \in \c{F}_{n-\v{1}}$.
In the latter case, the relation $e_n \in \c{F}_{n-\v{1}}$ is trivial since $\emptyset$ is always an element of any \mbox{$\sigma$-algebra}.
It~follows then that~\eqref{eqn:gundy_f_cond} holds for $f = a$ (we have $a_{0} = \Delta_{0} a = 0$ as a consequence of~\eqref{eqn:gundy_to_prove_2}).
Hence $\cbr{\abs{V a} > 0} \subseteq I_1 \x .. \x I_j \x [0,1)^{D-j}$ by assumption.
Consequently, we can write
\[
\int_{\del{I_1^{(r_1)}}^c \x .. \x \del{I_j^{(r_j)}}^c} \int_{A} \abs{V a}^p
\leq
\int_{\cbr{\del{V a} > 0}^c} \abs{V a}^p
=
0
\leq
C 2^{- \eta_1 r_1 - .. - \eta_j r_j}
,
\]
for whatever $\eta_1, .., \eta_D > 0$.
This is true because the set over which the integral is taken do not intersect the support of the function being integrated.

If $j = D - 1$ and $A = I_D \subseteq [0,1)$ is a Vilenkin interval we need to check also the condition~\eqref{eqn:quasi_local_Dd_2} of Theorem~\ref{thm:quasi_local_bounded}.
In this case, relation~\eqref{eqn:gundy_to_prove_1} holds true with the indicator $\1_{I_1 \x .. \x I_{D-1} \x I_D}$ instead of the indicator $\1_{I_1 \x .. \x I_{D-1} \x [0,1)}$, which will be noted separately in the proof of~\eqref{eqn:gundy_to_prove_1} and~\eqref{eqn:gundy_to_prove_2} below.
By assumption we then have $\cbr{\del{V a} > 0} \subseteq I_1 \x .. \x I_{D-1} \x I_{D}$, hence for any $r \in \N^D$
\[
\int_{\del{I_1^{(r_1)}}^c \x .. \x I_{D}^c} \abs{V a}^p
\leq
\int_{\cbr{\del{V a} > 0}^c} \abs{V a}^p
=
0
\leq
C 2^{- \eta_1 r_1 - .. - \eta_{D-1} r_{D-1}}
,
\]
for whatever $\eta_1, .., \eta_D > 0$, analogous to the above.

In order to finish the proof we need to verify the relations \eqref{eqn:gundy_to_prove_1} and \eqref{eqn:gundy_to_prove_2}.
To do this, let us study the support of $\E_n a$ for different values of $n \in \Z_+^D$.
Since $\E_n a$ is a constant on each atom $B$ of $\c{F}_n$ and equal to $1/\abs{B} \int_B a(x) \d x$ therein, we need only describe such $B$ that $\restr{\del{\E_n a}}{B} \not = 0$.

First, consider such $n$ that $n_d > N_d ~\forall 1 \leq d \leq j$.
Since $n_d \geq 0 = N_d, ~\forall d > j$, we have that $n \geq N$ holds, and thus either $B \subseteq S$ or $B \cap S = \emptyset$ must be true.
In the latter case $B$ does not intersect the support of $a$, thus we have $\restr{\del{\E_n a}}{B} = 1/\abs{B} \int_B a(x) = 0$, and then $\supp \del{\E_n a} \subseteq S$ follows.
Note that when $a$ is supported on $I_1 \x .. \x I_{D-1} \x I_D$ we may take $S = I_1 \x .. \x I_{D-1} \x I_D$ instead of $S = I_1 \x .. \x I_{D-1} \x [0,1)$ and the same argument will hold, proving that $\supp \del{\E_n a} \subseteq S = I_1 \x .. \x I_{D-1} \x I_D$ for $n > N$.

Now, consider such $n$ that $n_d > N_d ~\forall 1 \leq d \leq j$ does \emph{not} hold.
Then there exists a particular index $1 \leq d \leq D-j$ such that $n_d \leq N_d$.
In this case, for every atom $B = B_1 \x .. \x B_D$ of $\c{F}_n$ either $I_d \subseteq B_d$ or $I_d \cap B_d = \emptyset$ holds.
Then, denoting the set $B$ with its $d$-th dimension removed, by $\widetilde{B} = B_1 \x .. \x B_{d-1} \x B_{d+1} \x .. \x B_{D}$,
\[
\restr{\del{{\E}_n a}}{B}
\!=\!
\frac{1}{\abs{B}} \int_B\!\! a(x)
\!=\!
\frac{1}{\abs{B}} \int_{\widetilde{B}} \del{\int_{B_{d}}\!\!\! a(x) \d x_d} \d x_1 .. \d x_{d-1} \d x_{d+1} .. \d x_{D}
\!=\!
0,
\]
where we first used Fubini's theorem to switch the integration order and then, to show that the inner integral is zero, either $\supp a(.., x_{d-1}, \cdot, x_{d}, ..) \subseteq I_d$ or $\int a(x) \d x_d = 0$ when $I_d \cap B_d = \emptyset$ or $I_d \subseteq B_d$ respectively. Hence $\supp \del{\E_n a} \!\subseteq\! \emptyset$.

Since $\Delta_{n} a = \sum_{\varepsilon_1, ..,\varepsilon_D \in \cbr{0, 1}} (-1)^{\varepsilon_1 + .. + \varepsilon_D} \E_{n_1 - \varepsilon_1, .., n_D - \varepsilon_D} a$ (cf. Equation~\eqref{eqn:difference_definition}), we have $\supp \del{\Delta_{n} a} \subseteq \cup_{\varepsilon_1, ..,\varepsilon_D \in \cbr{0, 1}} \supp \del{\E_{n_1 - \varepsilon_1, .., n_D - \varepsilon_D} a}$.
Thus for $n$ with $n_d > N_d ~\forall 1 \leq d \leq j$ we have $\supp \del{\Delta_{n} a} \subseteq I_1 \x .. \x I_j \x [0,1)^{D-j}$, or even $\supp \del{\Delta_{n} a} \subseteq I_1 \x .. \x I_{D-1} \x I_D$ when $j = D - 1$ and $A = I_D$ is a Vilenkin interval.
When $n_d > N_d ~\forall 1 \leq d \leq j$ does not hold, we have $\supp \del{\Delta_{n} a} \subseteq \emptyset$ since $\supp \del{\E_{n_1 - \varepsilon_1, .., n_D - \varepsilon_D} a}$ is empty.
This verifies the relations \eqref{eqn:gundy_to_prove_1} and \eqref{eqn:gundy_to_prove_2} and so the proof is finished.
\end{proof}

\paragraph{A useful corollary}
Let us define \emph{the modified martingale difference operators}
\[
\Delta_{n, l} f
=
\sum\limits_{k \in \delta_{n,l}}
\innerprod{f}{v_k(\cdot)}
\ 
v_k(x),
\]
where the notation is analogous to~\eqref{eqn:difference_vilenkin_sum}, $l = (l_1, \ldots, l_D)$ with $1 \leq l_d < p^{(d)}_{n_d}$ and
\[ \label{eqn:delta_n_l_definition}
\begin{aligned}
\delta_{n, l}
=
&[l_1 P^{(1)}_{n_1-1}, (l_1 + 1) P^{(1)}_{n_1-1})
\\
\x
&[l_2 P^{(2)}_{n_2-1}, (l_2 + 1) P^{(2)}_{n_2-1})
\\
&\dots
\\
\x
&[l_D P^{(D)}_{n_D-1}, (l_D + 1) P^{(D)}_{n_D-1})
\subseteq \Z_+^D
.
\end{aligned}
\]
For the Walsh system we always have $l_d = 1$ and these operators coincide with~$\Delta_n$, but for the general Vilenkin systems they can be quite useful as they behave favorable under shifts for the general Vileknin systems (and $\delta_{n}$ do not), the fact that will be used in the geometric part of the proof of the main theorem.
Note that each dimension of $\delta_{n, l}$ corresponds to the numbers that have $l_d$ as their $n_d$-th digit (which is the most significant digit) in the mixed radix system corresponding to $\v{p}_d$.
Here we prove that Theorem~\ref{thm:gundy_Dd} holds with operators $\Delta_{n, l}$ instead of operators $\Delta_n$ in the assumption.
\begin{corollary} \label{thm:gundy_cor}
Fix $D \geq 2$ and consider a linear operator $V$ taking $D$-parameter Vilenkin martingales into measurable functions. Assume the following holds.
\1 $V: \c{L}^2 \to L^2$ is bounded.
\2 For any $D$-parameter martingale $f$ such that $f_{0} = 0$ and
\[
\Delta_{n, l} f = \1_{e_n} \Delta_{n, l} f, \t{where} e_n \in \c{F}_{n-\v{1}} \t{with} \del{n-\v{1}}_d = \max(n_d - 1, 0)
,
\]
we have $\cbr{\abs{V f} > 0} \subseteq \bigcup_{n \in \Z_+^D \setminus \{0\}} e_n$.
\0
Then $V: \c{H}^{p} \to L^{p}$ is bounded for any $p \leq 1$ and $V: \c{L}^r \to L^r$ is bounded for~$1 < r \leq 2$.
\end{corollary}
\begin{proof}
Consider a martingale $f$ for which $f_{0} = 0$ and
\[
\Delta_{n} f = \1_{e_n} \Delta_{n} f, \t{where} e_n \in \c{F}_{n-\v{1}}
.
\]
We claim that $\Delta_{n, l} f = \1_{e_n} \Delta_{n, l} f$.
If this is true, then direct application of Theorem~\ref{thm:gundy_Dd} proves the corollary.
To prove this, we show that for $l \not= s$, the functions $\Delta_{n, l} f$ and $\Delta_{n, s} f$ are $L^2$-orthogonal when restricted onto any atom of~$\c{F}_{n-\v{1}}$.
Since for any atom $B \in \c{F}_{n-\v{1}}$ such that $B \cap e_n = \emptyset$ we have ${\restr{\Delta_n f}{B} = 0}$, then, representing $\Delta_n f = \sum_l \Delta_{n, l} f$ and using orthogonality, we get that $\restr{\Delta_{n, l} f}{B} = 0$, which implies $\Delta_{n, l} f = \1_{e_n} \Delta_{n, l} f$.

\begin{samepage}
Now, we prove the orthogonality.
Fix an atom ${B = B_1 \x .. \x B_D \in \c{F}_{n-\v{1}}}$ and a pair of indices $l \not= s$.
We will show that $\int_B \del{\Delta_{n, l} f}(x) \del{\Delta_{n, s} f}(x) \d x = 0$.
Find an index $d$ be such that $l_d \not= s_d$.
Then $n_d > 0$, since for $n_d = 0$ there does not exist a pair of distinct indices $1 \leq l_d, s_d \leq p^{(d)}_0 = 1$ (there is only one $\Delta_{n_d, l_d} = \Delta_{n_d}$ in this case). Assume $d = 1$, without loss of generality.
We have
\<
\del{\Delta_{n, l} f}(x_1, .., x_D) &= {\sum}_{k = l_1 P^{(1)}_{n_1-1}}^{(l_1 + 1) P^{(1)}_{n_1-1}-1} a_k(x_2, .., x_D) v_k(x_1),
\\
\del{\Delta_{n, s} f}(x_1, .., x_D) &= {\sum}_{k = s_1 P^{(1)}_{n_1-1}}^{(s_1 + 1) P^{(1)}_{n_1-1}-1} b_k(x_2, .., x_D) v_k(x_1),
\>
where $v_k$ are one-parameter Vilenkin functions corresponding to the sequence~$\v{p}_1$.
\end{samepage}
We see that in order to prove orthogonality of $\Delta_{n, l} f$ and $\Delta_{n, s} f$ on atom $B$ it is enough to show the orthogonality of $\restr{v_k}{B_1}$ to $\restr{v_m}{B_1}$ for index $k$ in the set~$[l_1 P^{(1)}_{n_1-1}, (l_1 + 1) P^{(1)}_{n_1-1})$ and index $m$ in the set $[s_1 P^{(1)}_{n_1-1}, (s_1 + 1) P^{(1)}_{n_1-1})$.
To do this, write
\<
v_k(x_1) &= r_1^{\alpha_1} (x_1) r_2^{\alpha_2} (x_1) \cdot \ldots \cdot r_{n_1-1}^{\alpha_{n_1-1}} (x_1) r_{n_1}^{l_1} (x_1)
\\
v_m(x_1) &= r_1^{\beta_1} (x_1) r_2^{\beta2} (x_1) \cdot \ldots \cdot r_{n_1-1}^{\beta_{n_1-1}} (x_1) r_{n_1}^{s_1} (x_1),
\>
where $r_i$ are one-parameter generalized Rademacher functions, cf.~\eqref{eqn:rademacher_definition}.
The functions $r_i, 1 \leq i < n_1$ are constant on atom $B$, while $r_{n_1}^{l_1}$ and $r_{n_1}^{s_1}$ are orthogonal because
\[
\int_{B_1} r_{n_1}^{l_1} (x_1) \overline{r_{n_1}^{s_1} (x_1)} \d x_1
=
\int_{B_1} r_{n_1}^{l_1 - s_1} (x_1) \d x_1
=
\frac{1}{P^{(1)}_{n_1}} \sum_{r = 0}^{p^{(d)}_{n_1} - 1} e^{\frac{2 \pi i (l_1 - s_1) r}{p^{(d)}_{n_1}}}
=
0
.
\]
This proves the claim.
\end{proof}

\begin{remark}
Of course, both Theorem~\ref{thm:gundy_Dd} and Corollary~\ref{thm:gundy_cor} hold for operators taking $l^2$-valued martingales into measurable functions or taking scalar-valued martingales into $l^2$-valued functions.
The respective proofs do not change.
\end{remark}

\section{Boundedness of the auxiliary operators}
\label{sec:operators}

Here we introduce a version of the Littlewood--Paley square function associated with operators $\Delta_{n, l}$ instead of operators $\Delta_n$ and also the $D$-parameter Vilenkin analog of auxiliary operator $G$ from \cite{osipov2017}.
Boundedness of this square function and of the auxiliary operator is the subject of the two principal lemmas that we prove in this section.

\paragraph{A version of the square function} We start by defining a version of the Littlewood--Paley square function associated with operators $\Delta_{n, l}$ instead of $\Delta_n$.
\begin{definition}
An operator $S_m$ that maps (possibly $l^2$-valued) Vilenkin martingales or measurable functions into measurable functions by rule
\[
&S_m f = \del{\sum_{n \in \Z_+^D} \sum_{1 \leq l < p_n} \abs{\Delta_{n, l} f}^2}^{1/2},
&
&\t{with} p_n \stackrel{\text{def}}{=} (p^{(1)}_{n_1}, \ldots, p^{(D)}_{n_D}),
\]
will be referred to as \emph{the modified square function}.
\end{definition}

\begin{lemma} \label{thm:modified_s_bounded}
The modified square function $S_m: \c{L}^p \to L^p$ is bounded for ${1 < p \leq 2}$.
Same is true if $S_m$ is regarded as the operator $S_m: L^p \to L^p$.
Additionally, $S_m: \c{H}^p \to L^p$ is bounded for $p \leq 1$.
\end{lemma}
\begin{proof}
Consider the operator that maps martingales into $l^2_{\Z_+^D \x \N^D}$-valued functions such that
\[
\cbr[1]{\widetilde{S}_{m} f}_{n, l} =
\begin{cases}
\Delta_{n, l} f, & \t{for} n \in \Z_+^D \t{and} 1 \leq l < p_n, \\
0, & \t{for} n \in \Z_+^D \t{and} l \in \N^D \t{that do not satisfy} l < p_n.
\end{cases}
\]
Then, boundedness of $S_m: \c{L}^p \to L^p$ (resp. $S_m: \c{H}^p \to L^p$ for $p \leq 1$) follows from the boundedness of the operator $\widetilde{S}_m: \c{L}^p \to L^p(l^2_{\Z_+^D \x \N^D})$ (resp. the boundedness of the operator $\widetilde{S}_m: \c{H}^p \to L^p(l^2_{\Z_+^D \x \N^D})$ for $p \leq 1$), while the latter operator's boundedness follows directly from the $l^2$-valued version of Corollary~\ref{thm:gundy_cor}.
Indeed, $\norm[1]{\cbr[1]{\widetilde{S}_{m} f}_{n, l}}_{L^2(l^2)}^2 = \sum_{n \in \Z_+^D} \sum_{l = 1}^{p_n} \norm[1]{\Delta_{n, l} f}_{L^2}^2 = \norm[1]{f}_{\c{L}^2}^2$, ensuring condition 1 of Corollary~\ref{thm:gundy_cor} is fulfilled, while $\cbr[1]{\norm[1]{\cbr[1]{\widetilde{S}_{m} f}_{n, l}}_{l^2} > 0} = \bigcup_{n \in \Z_+^D \setminus \{0\}} e_n$ whenever $\Delta_{n, l} f = \1_{e_n} \Delta_{n, l} f$, $f_0 = 0$, and $e_n \in \c{F}_{n-\v{1}}$, ensuring condition 2 is fulfilled.
Boundedness of $S_m: L^p \to L^p$ follows from the one-to-one isometric correspondence between martingales and measurable functions for $1 < p \leq 2$.
\end{proof}

\paragraph{An auxiliary operator} Consider a family of multi-indices $\c{A} \subseteq \N \x \Z_+^D$.
We~denote elements of $\c{A}$ by $(k, n)$ where $k \in \N, n \in \Z_+^D$.
Consider also a family of indices $\cbr{a_{k, n}}_{(k, n) \in \c{A}} \subseteq \Z_+^D$ and a family of sets $\cbr{\Lambda_{k, n}}_{(k, n) \in \c{A}}$, such that $\Lambda_{k, n} \subseteq [1, p^{(1)}_{n_1}-1] \x .. \x [1, p^{(D)}_{n_D}-1]$ and $\cbr{a_{k, n} \dot{+} \delta_{n, l}}_{(k, n) \in \c{A}, l \in \Lambda_{k, n}}$ is a collection of pairwise non-intersecting subsets of $\Z_+^D$.

\begin{lemma} \label{thm:G_is_bounded}
Define operator $G$ mapping a function $h = \cbr{h_{k, n}}_{(k, n) \in \N \x \Z_+^D}$ taken from $L^p(l^2_{\N \x \Z_+^{D}})$, $1 < p \leq 2$ into
\[
(G h)(x_1, .., x_D)
=
\sum\limits_{(k, n) \in \c{A}, l \in \Lambda_{k, n}} v_{a_{k, n}}(x_1, .., x_D) \del{\Delta_{n, l} h_{k, n}}(x_1, .., x_D).
\]
Then $\norm{G h}_{L^p} \leq C \norm{h}_{L^p(l^2)}$.
Additionally, for $p \leq 1$ we have that $\norm{G h}_{L^p} \leq C \norm{h}_{\c{H}^p(l^2)}$.
Here $C$ depends on the exponent $p$ and on the sequences $\cbr[0]{p^{(D)}_{n_D}}_{n_D \in \Z_+}$.

\end{lemma}
\begin{proof}
Thanks to the one-to-one correspondence between martingales and measurable functions from $L^p$ for $1 < p \leq 2$, operator $G$ can be viewed as an operator mapping $l^2(\N \x \Z_+^D)$-valued martingales into measurable functions.

We will prove that $G$ satisfies the generalization of Corollary~\ref{thm:gundy_cor} to a case of $l^2(\N \x \Z_+^D)$-valued martingales.
As it was mentioned before, such generalization is straightforward.

First, it is obvious that $G$ is linear.
The Plancherel theorem and the fact that $\cbr{a_{k, n} \dot{+} \delta_{n, l}}_{(k, n) \in \c{A}, l \in \Lambda_{k, n}}$ is a collection of pairwise non-intersecting subsets of $\Z_+^D$ show that $G$ is bounded on $L^2$.

Finally, consider an $l^2(\N \x \Z_+^D)$-valued martingale $f$ for which $\Delta_0 f = 0$ and $\Delta_{n, l} f = \1_{e_n} \Delta_{n, l} f$ where $e_n \in \c{F}_{n - \v{1}}$.
Let us check that $\cbr{\abs{G f} > 0} \subseteq \bigcup_{n \in \Z_+^D \setminus \cbr{0}} e_n$.
Write
\<
\cbr{\abs{G f} > 0}
&\subseteq
\bigcup_{\substack{(k, n) \in \c{A}, \\ l \in \Lambda_{k, n}}}
    \cbr{
        \abs{
            v_{a_{k, n}}
            \Delta_{n, l} f_{k, n}
        }
        >
        0
    }
=
\bigcup_{\substack{(k, n) \in \c{A}, \\ l \in \Lambda_{k, n}}}
    \cbr{
        \abs{
            \Delta_{n, l} f_{k, n}
        }
        >
        0
    }
\\
&=
\bigcup_{\substack{(k, n) \in \c{A}, \\ l \in \Lambda_{k, n}}}
    \cbr{
        \abs{
            \1_{e_n}
            \Delta_{n, l} f_{k, n}
        }
        >
        0
    }
\subseteq
\bigcup_{(k, n) \in \c{A}}
    \cbr{
        \abs{
            \1_{e_n}
        }
        >
        0
    }
\\
&\subseteq
\bigcup\limits_{n \in \Z_+^D  \setminus \cbr{0}}
e_n
.
\>
This proves the claim.
\end{proof}

\section{A partition of an interval}
\label{sec:partition}

Here we will describe a method of dividing an interval $[a, b) \subseteq \Z_+$ of integers into subintervals that behave favorably under shifts induced by the operation $\dot{+}$ defined in Section~\ref{sec:prelim_vilenkin}.
This method, suggested by Tselishchev in~\cite{tselishchev2021}, is a key combinatorial component of the proof of the main theorem.

Throughout the section we will work exclusively in the one-parameter setting and assume a fixed (one-parameter) mixed-radix system $\v{p} = \cbr{p_n}_{n \in \N}$.

Consider an interval $I = [a, b) \subseteq \Z_+$.
We will find a decomposition
\< \label{eqn:final_partition}
&I
=
\del[1]{\tilde{J}_0 \cup \dots \cup \tilde{J}_t} \cup \del[1]{J_1 \cup \dots \cup J_{t-1}}
, &\text{where}
\\ \label{eqn:parition_properties}
&\tilde{J}_j \ \dot{-}\  a = \cup_{l \in \Lambda(\tilde{J}_j)} \delta_{j, l},
&J_j \ \dot{-}\  b = \cup_{l \in \Lambda(J_j)} \delta_{j, l}.
\>
Here $\dot{-}$ relates to $\dot{+}$ in the same way as the standard subtraction relates to the standard addition; when one of the operands of $\dot{-}$ is a set, it is simply applied element-wise yielding another set.
The intervals $\delta_{j, l}$ are as in~\eqref{eqn:delta_n_l_definition} and the particular index sets $\Lambda(\tilde{J}_j)$ and $\Lambda(J_j)$ will, though irrelevant to the proof of the main theorem, be described in the course of building the decomposition.

\begin{figure}
\centering
    \begin{tikzpicture}[every edge/.style={shorten <=1pt, shorten >=1pt}]
    \draw (0,0)  node [below] {} -- (10,0) node [below=2pt] {\footnotesize $1234$};
    \coordinate (p) at (0,2pt);
    \foreach \myprop/\mytext/\mylabel [count=\n] in {5.0/$\xcancel{J_4}$/$0$,2.0/$J_3$/$1000$,1.6/$J_2$/$1200$, 1.4/$J_1$/$1230$}
    \draw [decorate,decoration={brace,amplitude=4}] (p)  edge [draw] +(0,-4pt) edge node {\footnotesize \mylabel} +(0,-21pt)  -- ++(\myprop,0) coordinate (p) node [midway, above=2pt, anchor=south] {\mytext};
    \coordinate (p) at (1.75,-2pt);
    \foreach \myprop/\mytext/\mylabel [count=\n] in {0.5/$\tilde{J}_0$/$567$,0.75/$\tilde{J}_1$/$568$,0.875/$\tilde{J}_2$/$570$, 1.125/$\tilde{J}_3$/$600$}
    \draw [decorate,decoration={brace,amplitude=4,mirror}] (p)  edge [draw] +(0,4pt) edge node {\footnotesize \mylabel} +(0,-15pt)  -- ++(\myprop,0) coordinate (p) node [midway, below=8pt, anchor=north] {\mytext};
    \path (10,2pt) edge [draw]  ++(0,-4pt);
    \end{tikzpicture}
    \caption{The decomposition of interval $[a, b) = [567, 1234)$ with respect to the Vilenkin system corresponding to the $p_n \equiv 10$. We start by decomposing the interval~$[0, 1234)$ into $[1230, 1234) \cup [1200, 1230) \cup [1000, 1200) \cup [0, 1000)$. Then, since $a = 567 \in [0, 1000)$, we decompose additionally the interval~$[567, 1000)$ into  $\cbr{567} \cup [568, 570) \cup [570, 600) \cup [600, 1000)$. Note the digit-wise fashion of this decomposition.}
    \label{fig:partition}
\end{figure}
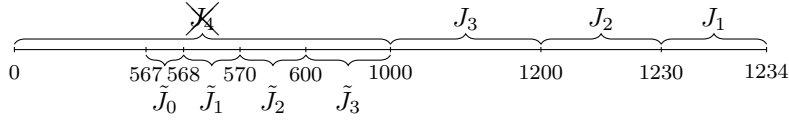

Consider mixed-radix representations of $a$ and $b$, as  in~\eqref{eqn:vilenkin_decomposition_integer},
\<
a &= \alpha_{k+1} \cdot P_{k} + \alpha_k P_{k-1} + \dots + \alpha_2 P_1 + \alpha_1,
\\
b &= \beta_{k+1} \cdot P_{k} + \beta_k P_{k-1} + \dots + \beta_2 P_1 + \beta_1,
\>
where $\alpha_j, \beta_j$ are integers with $0 \leq \alpha_k, \beta_k < p_k$ and $\beta_{k+1} > 0$ so that $\beta_{k+1}$ is the most significant digit of $b$.
Define $b_j = \beta_{k+1} P_k + \ldots + \beta_{j+1} P_j$ for $0 \leq j \leq k+1$.
This is just the number $b$ with its $j$ least significant digits set to zero. Define $a_j$ analogously.
We start by partitioning the interval $[0, b)$.
\[ \label{eqn:first_partition}
[0, b) = {\bigcup}_{j=1}^{k+1} J_j, ~~\t{where} J_j = [b_j, b_{j-1}).
\]
Note that $\abs{J_j} = \beta_j P_{j-1}$ and that $J_j$ consists of the numbers whose $j$-th digit $\gamma_j$ satisfies $0 \leq \gamma_j < \beta_j$ and whose less significant digits are arbitrary.
Naturally, there exists $1 \leq t \leq k+1$ such that $a \in J_t$, which implies $\alpha_j = \beta_j$ for $t+1 \leq j \leq k+1$ and $\alpha_t < \beta_t$.
We proceed to partition $[a, +\infty) \cap J_t$.
\< \label{eqn:second_partition}
&[a, +\infty) \cap J_t = {\bigcup}_{j=0}^{t} \tilde{J}_j, ~~\t{where}
\\
&\tilde{J}_j = [a_j + (\alpha_j + 1) P_{j-1}, a_j + p_j P_{j-1}) ~~\t{for} 1 \leq j \leq t-1,
\\
&\tilde{J}_t = [a_t + (\alpha_t + 1) P_{t-1}, a_t + \beta_t P_{t-1}), \qquad \tilde{J}_0 = \cbr{a}.
\>
Intervals $\tilde{J}_j$ consist of numbers whose $j$-th digit $\gamma_j$ satisfies $\alpha_j < \gamma_j < p_j$ for $1 \leq j \leq t - 1$ or $\alpha_j < \gamma_j < \beta_j$ for $j = t$, whose less significant digits are arbitrary and whose more significant digits coincide with those of $a$.
Note that both $J_j$ and $\tilde{J}_j$ are allowed to be empty.

Combining the partition of $[0, b)$ in~\eqref{eqn:first_partition} with the partition of $[a, +\infty) \cap J_t$ in~\eqref{eqn:second_partition}, we finally get the partition of $[a, b)$ in~\eqref{eqn:final_partition}, where the relations~\eqref{eqn:parition_properties} follow by direct computation.
The partition process is illustrated on Figure~\ref{fig:partition}.

\section{Proof of the main theorem}
\label{sec:proof}

Here we will finally prove the main theorem using the tools that were introduced thus far.
Let us first restate the theorem.
\vilenkinrdftwodim*

\begin{proof}

We will only present proof for $D = 2$ here.
While a particular $D$ does seriously affect the technique needed to prove Lemma~\ref{thm:G_is_bounded} (cf. the discussion in the preliminary Section~\ref{sec:prelim_boundedness}), it does not affect the combinatorial argument used to derive the main theorem from this lemma.
Note that symbol $D$ will be redefined in this proof to mean something not at all related to the number of parameters.

As in \cite{osipov2017}, we partition rectangles $I_k$ into fragments that behave well under shifts induced by operation $\dot{+}$.
This, together with Lemma~\ref{thm:G_is_bounded} and Lemma~\ref{thm:modified_s_bounded}, will allow us to prove the claim. 
Let
\[
I_k= I_k^1 \x I_k^2 = [a_k^{(1)}, b_k^{(1)}-1] \x [a_k^{(2)}, b_k^{(2)}-1]
.
\]

We build the partition of $I_k$ by forming the direct product of partitions of intervals $I_k^1$ and $I_k^2$, while partitioning these individual intervals as in Section~\ref{sec:partition}.
To partition a rectangle $I = I^1 \times I^2 = [a^{(1)}, b^{(1)}-1] \x [a^{(2)}, b^{(2)}-1] \subseteq \Z_+^2$ we partition each interval $I^i$ into the union of intervals $\tilde{J}^{(i)}_j$ and $J^{(i)}_j$ as by Section~\ref{sec:partition} and consider all direct products, yielding
\[
I
=
\del[3]{
    \bigcup\limits_{j}
    A_j
}
\cup
\del[3]{
    \bigcup\limits_{j}
    B_j
}
\cup
\del[3]{
    \bigcup\limits_{j}
    C_j
}
\cup
\del[3]{
    \bigcup\limits_{j}
    D_j
}
,
\]
where
\<
A_j
&=
\tilde{J}_{j^{\del{1}}}^{\del{1}}
\x
\tilde{J}_{j^{\del{2}}}^{\del{2}}
,
\qquad
B_j
=
J_{j^{\del{1}}}^{\del{1}}
\x
J_{j^{\del{2}}}^{\del{2}}
,
\\
C_j
&=
J_{j^{\del{1}}}^{\del{1}}
\x
\tilde{J}_{j^{\del{2}}}^{\del{2}}
,
\qquad
D_j
=
\tilde{J}_{j^{\del{1}}}^{\del{1}}
\x
J_{j^{\del{2}}}^{\del{2}}
,
\>
wherein a superscript refers to a partition of which of the segments $I^1$, $I^2$ does the indexed object belong.
The most important property of rectangles $A_j, B_j, C_j, D_j$ is that they can be shifted to become a union of rectangles $\delta_{k, l}$ by a constant (depending only on the number of parameters) number of shifts.
Define $a, b, c, d$ to be the vertices of rectangle $I$, that is
\[
a = \del[0]{a^{\del{1}}, a^{\del{2}}}
,~~
b = \del[0]{b^{\del{1}}, b^{\del{2}}}
,~~
c = \del[0]{b^{\del{1}}, a^{\del{2}}}
,~~
d = \del[0]{a^{\del{1}}, b^{\del{2}}}
,
\]
then, defining the index sets (with $\Lambda$ of $\tilde{J}^{(i)}_j$ and $J^{(i)}_j$ as in Section~\ref{sec:partition})
\<
&\Lambda(A_j) = \Lambda(\tilde{J}^{(1)}_j) \x \Lambda(\tilde{J}^{(2)}_j), &\Lambda(B_j) = \Lambda(J^{(1)}_j) \x \Lambda(J^{(2)}_j),
\\
&\Lambda(C_j) = \Lambda(J^{(1)}_j) \x \Lambda(\tilde{J}^{(2)}_j), &\Lambda(D_j) = \Lambda(\tilde{J}^{(1)}_j) \x \Lambda(J^{(2)}_j),
\>
we have, with notation $\spec f = \Set{n \in \Z_+^2}{\innerprod{f}{v_n} \not= 0}$, that
\<
\label{eqn:spec_prop_1}
{\sum}_{l \in \Lambda(A_j)} \Delta_{j, l}
v^{-1}_{a}
f
&=
v^{-1}_{a}
f
, ~~~\t{if}~ \spec{f} \subseteq A_j,
\\
\label{eqn:spec_prop_2}
{\sum}_{l \in \Lambda(B_j)} \Delta_{j, l}
v^{-1}_{b}
f
&=
v^{-1}_{b}
f
, ~~~\t{if}~ \spec{f} \subseteq B_j,
\\
\label{eqn:spec_prop_3}
{\sum}_{l \in \Lambda(C_j)} \Delta_{j, l}
v^{-1}_{c}
f
&=
v^{-1}_{c}
f
, ~~~\t{if}~ \spec{f} \subseteq C_j,
\\
\label{eqn:spec_prop_4}
{\sum}_{l \in \Lambda(D_j)} \Delta_{j, l}
v^{-1}_{d}
f
&=
v^{-1}_{d}
f
, ~~~\t{if}~ \spec{f} \subseteq D_j.
\>
This behavior under shifts is the basis of the following argument.

Let us similarly partition each $I_k$, adding the additional index $k$ to all objects that arise from this construction.
Then $f_k$ can be represented as sum
\[
f_k
=
\sum\limits_{j} f_{k, j}^A
+
\sum\limits_{j} f_{k, j}^B
+
\sum\limits_{j} f_{k, j}^C
+
\sum\limits_{j} f_{k, j}^D
,
\]
where $\spec f_{k, j}^A \subseteq A_{k, j}$, $\spec f_{k, j}^B \subseteq B_{k, j}$, $\spec f_{k, j}^C \subseteq C_{k, j}$, $\spec f_{k, j}^D \subseteq D_{k, j}$.
Define
\<
g_{k, j}^A
=
v^{-1}_{a_k}
f_{k, j}^A
,~~~
g_{k, j}^B
=
v^{-1}_{b_k}
f_{k, j}^B
,~~~
g_{k, j}^C
=
v^{-1}_{c_k}
f_{k, j}^C
,~~~
g_{k, j}^D
=
v^{-1}_{d_k}
f_{k, j}^D
.
\>
Then $f_k$ may be represented as follows
\[
\begin{aligned}
f_k =
&v_{a_k}
\sum_{j}
\sum_{l \in \Lambda(A_{k, j})}
\Delta_{j, l} g_{k, j}^A
+
v_{b_k}
\sum_{j}
\sum_{l \in \Lambda(B_{k, j})}
\Delta_{j, l} g_{k, j}^B
\\
+
&v_{c_k}
\sum_{j}
\sum_{l \in \Lambda(C_{k, j})}
\Delta_{j, l} g_{k, j}^C
+
v_{d_k}
\sum_{j}
\sum_{l \in \Lambda(D_{k, j})}
\Delta_{j, l} g_{k, j}^D
.
\end{aligned}
\]

Defining $h_{k, n} = g_{k, j(n)}^{S(n)}$, $\Lambda_{k, n} = \Lambda(S(n)_{k, j(n)})$ and $a_{k, n} = s(n)_k$ with appropriate enumeration $j(n) \in \Z_+^2$, $S(n) \in \cbr{A, B, C, D}$ and $s(n) \in \cbr{a, b, c, d}$, we see that $\sum_{k} f_k = G \del[1]{\cbr{h_{k, n}}_{k \in \N, n \in \Z_+^2}}$ and the assumptions of Lemma~\ref{thm:G_is_bounded} are fulfilled (this is justified by nothing other than the properties~\eqref{eqn:spec_prop_1}--\eqref{eqn:spec_prop_4}).
Applying this lemma and then the triangle inequality, gives us
\< \label{eqn:G_application_result}
&\norm{
    \sum_{k} f_k
}_{L^p}
\!\!\!\!\!\lesssim
\Bigg\lVert
    \Bigg(
        \!\!\!
        \sum\limits_{k}
        \!\!
        \Big(
            \!
            \sum\limits_{j}
                \abs{
                    g_{k, j}^A
                }^2
            \!\!+\!\!
            \sum\limits_{j}
                \abs{
                    g_{k, j}^B
                }^2
            \!\!+\!\!
            \sum\limits_{j}
                \abs{
                    g_{k, j}^C
                }^2
            \!\!+\!\!
            \sum\limits_{j}
                \abs{
                    g_{k, j}^D
                }^2
        \Big)
    \Bigg)^{\!\!\frac{1}{2}}
\Bigg\rVert_{L^p}
\\
\label{eqn:after_G_1}
&\lesssim
\Bigg\lVert
    \Bigg(
        \!\!
        \sum\limits_{k, j}
            \sum_{l \in \Lambda(A_{k, j})}
            \!\!
            \abs{
                \Delta_{j, l} g_{k, j}^A
            }^2
    \Bigg)^{\frac{1}{2}}
\Bigg\rVert_{L^p}
\!\!\!\!+\!
\Bigg\lVert
    \Bigg(
        \!\!
        \sum\limits_{k, j}
            \sum_{l \in \Lambda(B_{k, j})}
            \!\!
            \abs{
                \Delta_{j, l} g_{k, j}^B
            }^2
    \Bigg)^{\frac{1}{2}}
\Bigg\rVert_{L^p}
\\
\label{eqn:after_G_2}
&+
\Bigg\lVert
    \Bigg(
        \!\!
        \sum\limits_{k, j}
            \sum_{l \in \Lambda(C_{k, j})}
            \!\!\!
            \abs{
                \Delta_{j, l} g_{k, j}^C
            }^2
    \Bigg)^{\frac{1}{2}}
\Bigg\rVert_{L^p}
\!\!\!\!+
\Bigg\lVert
    \Bigg(
        \!\!
        \sum\limits_{k, j}
            \sum_{l \in \Lambda(D_{k, j})}
            \!\!\!
            \abs{
                \Delta_{j, l} g_{k, j}^D
            }^2
    \Bigg)^{\frac{1}{2}}
\Bigg\rVert_{L^p}
.
\>
Since $G$ is boundeded from $\c{H}^p$ to $L^p$ for $p \leq 1$, we have in this case the inequality similar to~\eqref{eqn:G_application_result} with $\c{H}^p(l^2_{\N \x \Z_+^2})$-norm instead of the $L^p(l^2_{\N \x \Z_+^2})$-norm of the sequence $\cbr{h_{k, n}}$ on the right hand side.
By noting that, thanks to the properties~\eqref{eqn:spec_prop_1}--\eqref{eqn:spec_prop_4}, this $\c{H}^p$-norm coincides with the $L^p$-norm, we see that this inequality holds for $p \leq 1$ just as well.

The next step is to estimate each term of~\eqref{eqn:after_G_1},~\eqref{eqn:after_G_2} separately.
Consider, for example, the third term.
Write
\<
&v^{-1}_{c_k} f_k
=
v_{a_k \dot{-} c_k}
\sum\limits_{j}
g_{k, j}^A
+
v_{b_k \dot{-} c_k}
\sum\limits_{j}
g_{k, j}^B
+
\sum\limits_{j}
g_{k, j}^C
+
v_{d_k \dot{-} c_k}
\sum\limits_{j}
g_{k, j}^D
\\
&=
v_{a_k \dot{-} c_k}
\!
\sum\limits_{j}
g_{k, j}^A
\!+\!
v_{b_k \dot{-} c_k}
\!
\sum\limits_{j}
g_{k, j}^B
\!+\!
\sum\limits_{j}
\!\!
\sum_{l \in \Lambda(C_{k, j})}
\!\!\!\!\!\!
\Delta_{j, l}
g_{k, j}^C
\!+\!
v_{d_k \dot{-} c_k}
\!
\sum\limits_{j}
g_{k, j}^D.
\>
We note that $\Delta_{j, l} v^{-1}_{c_k} f_k = \Delta_{j, l} g_{k, j}^C$, hence in the decomposition
\[
v^{-1}_{c_k} f_k
=
{\sum}_{n \in \Z_+^2} {\sum}_{\v{1} \leq l \leq p_n} \Delta_{n, l} v^{-1}_{c_k} f_k
,
\]
functions $\Delta_{j, l} g_{k, j}^C$ are among the right hand side terms.
It follows then that
\[
\sum\limits_{j}
\sum_{l \in \Lambda(C_{k, j})}
\abs{
    \Delta_{j, l} g_{k, j}^C
}^2
\leq
\sum\limits_{n \in \Z_+^2}
\sum_{\v{1} \leq l \leq p_n}
\abs{
    \Delta_{n, l} v^{-1}_{c_k} f_k
}^2
=
\del[2]{S_m(v^{-1}_{c_k} f_k)}^2
,
\]
where $S_m$ is the modified Littlewood--Paley square function defined in Section~\ref{sec:operators}.
By leveraging the $l^2$-valued analog of Lemma~\ref{thm:modified_s_bounded}, we have\footnote{Equations~\eqref{eqn:s_m_1}--\eqref{eqn:s_m_2} here trivially extend to $p \leq 1$, but Equation~\eqref{eqn:s_m_3} does not.}
\<
\label{eqn:s_m_1}
&\norm[2]{
    \del[2]{
        \sum_{k}
        \sum_{j}
            \sum_{l \in \Lambda(C_{k, j})}
            \abs{
                \Delta_{j, l}
                g_{k, j}^C
            }^2
    }^{1/2}
}_{L^p}
\\
\label{eqn:s_m_2}
\leq
&\norm[2]{
    \del[2]{
        \sum\limits_{k}
            \sum\limits_{n \in \Z_+^2}
                \sum_{l \in \Lambda(C_{k, j})}
                \abs{
                    \Delta_{n, l} v^{-1}_{c_k} f_k
                }^2
    }^{1/2}
}_{L^p}
\!\!\!\!=
\norm{
    S_m\del{
        \cbr{
        v^{-1}_{c_k} f_k
        }_{k \in \N}
    }
}_{L^p}
\\
\label{eqn:s_m_3}
\lesssim
&\norm[2]{
    \del[2]{
        \sum\limits_{k}
        \abs{
            v^{-1}_{c_k} f_k
        }^2
    }^{1/2}
}_{L^p}
=
\norm[2]{
    \del[2]{
        \sum\limits_{k}
        \abs{
            f_k
        }^2
    }^{1/2}
}_{L^p}
.
\>

Analogous to~\eqref{eqn:s_m_1}--\eqref{eqn:s_m_3} we can bound each of the four terms in \eqref{eqn:after_G_1} and \eqref{eqn:after_G_2}.
Gathering these inequalities together, we finally obtain
\[
\norm[2]{\sum\limits_{k} f_k}_{L^p}
\lesssim
\norm[2]{
    \del[2]{
        \sum\limits_{k}
        \abs{
            f_k
        }^2
    }^{1/2}
}_{L^p}
,
\]
which proves the claim.
\end{proof}

\newpage

\section{Some corollaries and extensions}
\label{sec:extensions}

Here we discuss some corollaries and extensions of the main theorem.

\subsection{One-parameter Rubio de Francia inequality for the Walsh system with unusually defined intervals} \label{sec:exotic_intervals}

One of the rather general views upon Rubio de Francia inequality may be through abstract harmonic analysis.
Consider a locally compact abelian group~$G$, denote its Potryagin dual by $\hat{G}$ and fix some partial order $\leq_{\hat{G}}$ on $\hat{G}$. Then,
\1 Lebesgue spaces $L^p(G)$ may be defined through the Haar measure on $G$,
\2 operators $M_I$ for sets $I \subseteq \hat{G}$ may be defined as the abstract Fourier transform multipliers with symbol $\1_{I}$,
\3 generalized intervals may be defined as $I = [a, b] = \Set{x \in \hat{G}}{a \leq_{\hat{G}} x \leq_{\hat{G}} b}$,
\0
substituting these into the classical inequality~\eqref{eqn:rdf_classic}, we may formulate a relation that can be called the Rubio de Francia inequality on $G$ with fixed order $\leq_{\hat{G}}$ on the dual group.
Whether the inequality holds for a particular $G$ is a non-trivial question, but its sides are well-defined in this general setting.

Taking this point of view, it may be tempting to assume that since a multi-parameter Vilenkin group is always isomorphic to a one-parameter Vilenkin group, we may prove Rubio de Francia inequality for the former by reducing it to Rubio de Francia inequality for the latter.
Although this path can be taken, it is not trivial, because the group isomorphism must also preserve the generalized interval structure (through the order of the dual group).
Proving Rubio de Francia inequality for the multi-parameter Vilenkin system by means of such group-theoretic machinery is out of scope of this work.
However, by running this machinery the other way around, we may obtain Rubio de Francia inequality for the one-parameter Vilenkin system with unusual alternative notions of the interval, which we demonstrate here on a simple example.

Consider the Cantor Dyadic Group and denote it by $C$.
This is a Vilenkin group corresponding to the sequence $\cbr{p_n}_{n \in \N}$ with $p_n \equiv 2$.
This is a compact abelian group.
Moreover, it may be identified with the interval $[0,1)$ by a measure-preserving transform.
For this reason, we may consider the characters of this group, elements of $\hat{C}$, as functions acting on $[0,1)$, namely the Walsh functions, which are a special case of the Vilenkin functions.
As by Section~\ref{sec:prelim_vilenkin}, Walsh functions may be enumerated by non-negative integers $\Z_+$ and ordered accordingly.
This yields Rubio de Francia inequality for the one-parameter Walsh system that was proved by Osipov in \cite{osipov2017}.

Consider now the square $C \x C$ of the Cantor Dyadic Group and identify it with $[0,1) \x [0,1)$ by a measure-preserving transform.
Its characters can be interpreted as two-parameter Walsh functions indexed by $\Z_+ \x \Z_+$ with the partial order $(n, m) \leq_{\Z_+ \x \Z_+} (k, l) \iff n \leq k ~\&~ m \leq l$ that gives rise to the generalized intervals that coincide with the rectangles of $\Z_+ \x \Z_+$.
Rubio de Francia inequality corresponding to this setting was proved by the author in~\cite{borovitskiy2022littlewood} and is a special case of the main theorem of this paper.

Consider the isomorphism $\Phi: C \x C \to C$ given by
\[
\Phi((x_1, x_2, \dots), (y_1, y_2, \dots)) = (x_1, y_1, x_2, y_2, \dots).
\]
It is easy to check that it preserves the group structure, topological structure and the measure.
It also defines an isomorphism between the dual groups: a character $w_{n, m}(x_1, x_2)$ of $C \x C$ corresponds to the character $w_{\Psi(n, m)}(\Phi(x_1, x_2))$ of $C$.
It can be checked that
\[
\Psi(n, m)
=
\Psi\del[1]{{\sum}_{k=1}^\infty n_k \cdot 2^{k-1}, {\sum}_{k=1}^\infty m_k \cdot 2^{k-1}}
=
{\sum}_{k=1}^\infty l_k \cdot 2^{k-1},
\]
where $l_k = n_k$ for odd $k$ and $l_k = m_k$ for even $k$.

This isomorphism takes the natural partial ordering of $\Z_+ \x \Z_+$ into an exotic ordering of $\Z_+$.
Thus, from the two-parameter Rubio de Francia inequality, there follows the one-parameter inequality with intervals replaced by the generalized intervals corresponding to this exotic ordering.
For instance, a set $\cbr{5, 16, 17, 20}$ will be a valid generalized interval, since it is the image of the interval $[(3, 0), (6, 0)] = [3, 6] \x \cbr{0} \subseteq \Z_+^2$ under the isomorphism $\Psi$.

Although here we restrict ourselves to this particular example, it is quite clear that this method enables one to show a plethora of exotic inequalities of this kind.

\subsection{A no-go result for arbitrary partitions of the set~$\Z_+^D$} \label{sec:impossible}

It is very natural to ask whether or not the multi-parameter Rubio de Francia inequality for Vilenkin systems studied in this paper may be extended from the case of partitions of the set~$\Z_+^D$ into arbitrary rectangles to partitions into arbitrary sets.
Here we present an argument showing that this is impossible.

For simplicity, let us only consider the case of the Walsh system.
We will use the ideas from Section~\ref{sec:exotic_intervals} and take inspiration in the argument of T.~Tao\footnote{The argument is unpublished but available online at~\cite{tao2021}.} that settles the same question for the case of the group $\T^D$.

Consider some $1 < p <2$, $D \in \N$ and suppose that
\[ \label{eqn:impossible:main_ineq}
\norm[1]{{\sum}_k f_k}_{L^p([0,1)^D)} \leq C_D \norm[1]{\cbr{f_k}}_{L^p([0,1)^D,\, l^2)}
\]
for all $f_k$ with Walsh--Fourier spectra in arbitrary sets $I_k \subseteq \Z_+^D$ with $\cup_k I_k = \Z_+^D$.

By considering only sets of form $I_k = I^0_k \x \Z_+^{D-1}$ we get the one-parameter Rubio de Francia inequality for the Walsh system and arbitrary partitions of~$\Z_+$.
Thus, without loss of generality, we may assume $D=1$.

Using the machinery from~\ref{sec:exotic_intervals} it is easy to show that our assumptions imply the Rubio de Franica inequality for arbitrary rectangles in~$\Z_+^D$ with the constant~$C_1$ independent of dimension $D$: the rectangles in $\Z_+^D$ correspond to some exotic intervals in $\Z_+$ and for them the inequality is true by assumption.

Let us show that from this we can infer the one-parameter Rubio de Francia inequality for the Walsh system with the constant $C_1 \leq 1$.
Of course it suffices to verify this only when the number of non-zero functions $f_k$ is finite.
Obviously, each of these functions is a finite linear combination of Walsh functions (an analog of a trigonometric polynomial) and the spectra of $f_k$ are intervals in $\Z_+$.

Denote the number of non-zero $f_k$ by $N$.
Then $k$ runs through the finite set~$1, \ldots, N$.
Consider the functions $f_{\v{k}}$, $\v{k} \in \cbr{1, \ldots, N}^D$ defined by the expression $f_{\v{k}}(x_1, \ldots, x_D) = f_{\v{k}_1}(x_1) \cdot \ldots \cdot f_{\v{k}_D}(x_D)$.
It is easy to see that the Walsh--Fourier spectra of $f_{\v{k}}$ are non-intersecting rectangles.

Substituting $f_{\v{k}}$ into the Rubio de Francia inequality for arbitrary rectangles in $\Z_+^D$ with the constant $C_1$ that we have obtained above and using, respectively,
\[
\sum_{\v{k} \in \cbr{1, \ldots, N}^D} \!\!\!\!\!\!\!\! f_{\v{k}}(x_1, \ldots, x_D)
&=
\del[2]{\,\sum_{k_1 = 1}^N f_k(x_1)} \cdot \ldots \cdot \del[2]{\,\sum_{k_D = 1}^N f_k(x_D)},
\\
\sum_{\v{k} \in \cbr{1, \ldots, N}^D} \!\!\!\!\!\!\!\! \abs{f_{\v{k}}(x_1, \ldots, x_D)}^2
&=
\del[2]{\,\sum_{k_1 = 1}^N \abs{f_k(x_1)}^2} \cdot \ldots \cdot \del[2]{\,\sum_{k_D = 1}^N \abs{f_k(x_D)}^2},
\]
to transform the left- and right-hand sides of the inequality, we see
\[
\norm[1]{{\sum}_k f_k}_{L^p([0,1))}^D \leq C_1 \norm[1]{\cbr{f_k}}_{L^p([0,1),\, l^2)}^D.
\]
It immediately follows that $C_1$ may be taken such that $C_1 \leq 1$.

Now we get the contradiction with the fact that Rubio de Francia inequality and even Littlewood--Paley inequality for the Walsh system do not hold with the constant $C_1 \leq 1$.
Indeed, the exact constant in the Littlewood--Paley inequality for the Walsh system is known to be greater than one.
See, for instance, the book \cite[Theorem 8.8]{oskekowski2012} as a reference on the subject.
Let us note that the inequality $C_1 \leq 1$ may be disproved without using the nontrivial results on exact constants, similarly to how it is done by Tao~\cite{tao2021}.

\subsection{The case of exponents $p \leq 1$}

Looking carefully at the remarks made during the proof of the main theorem in Section~\ref{sec:proof}, we can see that there can be formulated a weaker version of Rubio de Francia inequality for exponents $p \leq 1$.
Formally, the following is true.

\begin{theorem}
Let $I_k = I_k^1 \x \ldots \x I_k^D$ be a family of disjoint rectangles in~$\Z_+^D$.
Let $f_k: [0,1)^D \to \R$ be some functions with Vilenkin--Fourier spectrum in $I_k$, where $v_{n_d}$ are Vilenkin functions corresponding to a bounded Vilenkin system (which may be different for different values of $d$).
Then
\[
\norm[2]{{\sum}_k f_k}_{L^p}
\leq
C
\sum_{a = 1}^{2^D} \norm[2]{\cbr{v^{-1}_{\tau_{a, k}} f_k}_{k \in \N}}_{\c{H}^p(l^2_\N)}
,
\qquad p \leq 1
,
\]
where $\tau_{a, k}$ is the $a$-th corner of rectangle $I_k$, the constant $C$ does not depend on the choice of rectangles $\cbr{I_k}$ or functions $\cbr{f_k}$  and $\c{H}^p$ is the martingale Hardy space defined in Section~\ref{sec:prelim_spaces}.
\end{theorem}
\begin{proof}[Proof outline]
Repeat the proof from Section~\ref{sec:proof} up to (and including) Equation~\ref{eqn:s_m_2}.
Everything there remains valid for $p \leq 1$, as noted in the proof itself.
However, since $S_m$ is only bounded from $\c{H}^p$ to $L^p$ for $p \leq 1$, we cannot obtain Equation~\eqref{eqn:s_m_3}.
Instead, in this case we have
\[
\norm[2]{
    \del[2]{
        \sum_{k}
        \sum_{j}
            \sum_{l \in \Lambda(C_{k, j})}
            \abs{
                \Delta_{j, l}
                g_{k, j}^C
            }^2
    }^{1/2}
}_{L^p}
\leq
\norm[2]{\cbr{v^{-1}_{c_k} f_k}_{k \in \N}}_{\c{H}^p(l^2_\N)}
.
\]
This, repeated for every term in Equations~\eqref{eqn:after_G_1}--\eqref{eqn:after_G_2}, proves the stated claim.
\end{proof}
It is unknown to the author whether the corresponding result with Lebesgue norm on the right hand side holds true for $p \leq 1$.
Although an analogous statement has been known for the trigonometric Fourier series (or Fourier transform) for some time \cite{kislyakov2005, osipov2010a, osipov2010b}, to the best knowledge of the author, it remains an open problem even for the simplest case of the one-parameter Walsh system.

\section*{Acknowledgements}
Sections 1--3 of this research were financially supported by the Russian Science Foundation grant N\textsuperscript{\underline{o}}18-11-00053 (\emph{https://rscf.ru/project/18-11-00053/}).
The rest of the paper was financially supported by the Foundation for the Advancement of Theoretical Physics and Mathematics ``BASIS''.

The author would like to thank Prof. Sergei Kislyakov, Dr. Nikolay Osipov and Dr. Anton Tselishchev for useful discussions.
The author is also grateful to Prof. Ferenc Weisz who provided some insights and literature directions on the subject of martingale Hardy spaces.

\bibliography{references}

\end{document}